\documentclass[11pt]{amsart}
\usepackage{amssymb}
\usepackage{amsthm}
\usepackage{amsmath}
\usepackage{latexsym}
\usepackage{comment}
\usepackage{hyperref}
\usepackage{verbatim}
\usepackage{xypic}
\usepackage{graphicx}
\usepackage[utf8]{inputenc}
\usepackage[margin=1in]{geometry}
\newtheorem{thm}{Theorem}[section]
\newtheorem{prop}[thm]{Proposition}

\newtheorem{lemma}[thm]{Lemma}
\newtheorem{cor}[thm]{Corollary}
\newtheorem{dfn}[thm]{Definition}

\theoremstyle{definition} 

\newtheorem{hyp}[thm]{Hypothesis}
\theoremstyle{definition} \newtheorem{rmk}[thm]{Remark}

\newcommand{\zz}{\mathbb{Z}}
\newcommand{\ff}{\mathbb{F}}

\newcommand{\Gal}{\mathrm{Gal}}

\newcommand{\Sp}{\mathrm{Sp}}

\newcommand{\End}{\mathrm{End}}
\newcommand{\Aut}{\mathrm{Aut}}

\graphicspath{{pics/}}

\title{Lifting images of standard representations of symmetric groups}
\author{Jeffrey Yelton}

\begin{document}

\maketitle

\begin{abstract}

We investigate closed subgroups $G \subseteq \Sp_{2g}(\zz_2)$ whose modulo-$2$ images coincide with the image $\mathfrak{S}_{2g + 1} \subseteq \Sp_{2g}(\ff_2)$ of $S_{2g + 1}$ or the image $\mathfrak{S}_{2g + 2} \subseteq \Sp_{2g}(\ff_2)$ of $S_{2g + 2}$ under the standard representation.  We show that when $g \geq 2$, the only closed subgroup $G \subseteq \Sp_{2g}(\zz_2)$ surjecting onto $\mathfrak{S}_{2g + 2}$ is its full inverse image in $\Sp_{2g}(\zz_2)$, while all subgroups $G \subseteq \Sp_{2g}(\zz_2)$ surjecting onto $\mathfrak{S}_{2g + 1}$ are open and contain the level-$8$ principal congruence subgroup of $\Sp_{2g}(\zz_2)$.  As an immediate application, we are able to strengthen a result of Zarhin on $2$-adic Galois representations associated to hyperelliptic curves.  We also prove an elementary corollary concerning even-degree polynomials with full Galois group.

\end{abstract}

\section{Introduction} \label{sec1}

Let $d \geq 3$ be an integer, and let $S_d$ denote the symmetric group on $d$ points.  We have the well-known \textit{standard representation} of $S_d$ over $\ff_2$, which is defined as in \cite[\S2.2]{wagner1976faithful}.  Namely, given an $\ff_2$-vector space $W$ of dimension $d$ and an ordered basis $\{e_1, ... , e_d\}$, there is an obvious left action of $S_d$ defined by $^{\sigma}\!e_i = e_{\sigma(i)}$ for $\sigma \in S_d$ and $1 \leq i \leq d$.  The subspace 
$$W_0 := \{w \in W \ | \ \sum_{i = 1}^{d} e_i^*(w) = 0\},$$
 where each $e_i^* : W \to \ff_2$ is the $i$th coordinate map, is clearly invariant under this action of $S_d$.  Moreover, if $d$ is even, the element $w_0 \in W$ whose coordinates with respect to our fixed basis are all $1$'s lies in $W_0$ and is fixed by $S_d$, thus inducing a representation of $S_d$ on $W_0 / \langle w_0 \rangle$.  The standard representation when $d$ is odd (resp. when $d$ is even) is the induced representation of $S_d$ on $\bar{V} := W_0$ (resp. on $\bar{V} := W_0 / \langle w_0 \rangle$).  Note that in both cases, $\bar{V}$ has dimension $2g$, where $g = \lfloor (d - 1) / 2 \rfloor$.

Since any element of $W_0$ can be written as $\Sigma_{i \in I} e_i$ where $I \subseteq \{1, ... , d\}$ is an even-cardinality subset, each element of $\bar{V}$ is determined by such a subset $I$.  We will therefore denote the element of $\bar{V}$ given by the subset $I$ by $v_I$, keeping in mind that when $d$ is odd (resp. even), we have $v_I = v_{I'}$ if and only if $I = I'$ (resp. $I = I'$ or $I = \{1, ... , d\} \smallsetminus I'$).  We have an alternating bilinear pairing $\langle \cdot, \cdot \rangle : \bar{V} \times \bar{V} \to \ff_2$ given by setting $\langle v_I, v_{I'} \rangle$ to be the cardinality of the intersection $I \cap I'$ modulo $2$.  It is immediate to check that the action of $S_d$ respects this pairing, and therefore, each $\sigma$ acts on $\bar{V}$ as a symplectic automorphism of $\bar{V}$ with respect to $\langle \cdot, \cdot \rangle$.  We therefore write $\rho : S_d \to \Sp(\bar{V})$ to denote the standard representation of $S_d$, where $\Sp(\bar{V}) := \{\varphi \in \Aut(\bar{V}) \ | \ \langle ^{\varphi}v, ^{\varphi}\!w \rangle = \langle v, w \rangle, \ v, w \in \bar{V}\}$ is the group of \textit{symplectic automorphisms} with respect to the alternating form $\langle \cdot, \cdot \rangle$.  We denote the image of $\rho$ by $\mathfrak{S}_d \subseteq \Sp(\bar{V})$.  It is straightforward to verify directly from the construction above that $\rho$ is injective if and only if $d \neq 4$; meanwhile, we see by comparing orders using the formula for $\#\Sp_{2g}(\ff_2)$ found in \cite[\S III.6]{artin2011geometric} that $\rho$ is surjective if and only if $d \in \{3, 4, 6\}$.  So in particular, $\Sp(\bar{V}) = \mathfrak{S}_3 = \mathfrak{S}_4 \cong S_3$ when $g = 1$ and $\Sp(\bar{V}) = \mathfrak{S}_6 \cong S_6$ when $g = 2$.

Given a free $\zz_2$-module $V$ of rank $2g$ and an alternating $\zz_2$-bilinear pairing on $V$, on identifying $\bar{V}$ with $V / 2V$ and assuming that the pairing on $V$ induces the pairing $\langle \cdot, \cdot \rangle$ on $\bar{V}$, it is natural (for instance, in the abelian variety context discussed below) to consider the properties of subgroups $G \subseteq \Sp(V)$ whose images modulo $2$ coincide with $\mathfrak{S}_d$.  We will now present our main result, which classifies these liftings.

\subsection{Main result} \label{sec1.1}

In the below statement and for the rest of the paper, we adopt the following notations.  By mild abuse of notation, we write $\langle \cdot, \cdot \rangle$ both for the pairing on $V$ and for the pairing on $V / 2V$.  For each integer $n \geq 1$, we denote the reduction-modulo-$2^n$ map by $\varpi_{2^n} : \Sp(V) \to \Sp(V / 2^n V)$; if $m \geq n$ is another integer, we write $\varpi_{2^m \to 2^n} : \Sp(V / 2^m V) \twoheadrightarrow \Sp(V / 2^n V)$ for the induced reduction-modulo-$2^n$ map.  These maps are well known to be surjective (see for instance the proof of \cite[Lemma 5.16]{deligne1969irreducibility}).  For each $n \geq 1$, we write $\Gamma(2^n) \lhd \Sp(V)$ for the kernel of $\varpi_{2^n}$.

\begin{thm} \label{thm main}

With $V$ and $\mathfrak{S}_d \subseteq \Sp(V / 2V)$ as above for some fixed $d \geq 3$, let $G \subseteq \Sp(V)$ be a closed subgroup (under the $2$-adic topology), and assume that we have $\varpi_2(G) \supseteq \mathfrak{S}_d$.

a) If $d \geq 6$ is even, then $G$ coincides with the inverse image of $\varpi_2(G)$ under the reduction-modulo-$2$ map $\varpi_2$; it is thus an open subgroup of $\Sp(V)$ which contains $\Gamma(2) \lhd \Sp(V)$.

b) If $d \geq 5$ is odd, then $G$ contains $\Gamma(8) \lhd \Sp(V)$ and is therefore again an open subgroup of $\Sp(V)$.  If moreover we have $\varpi_2(G) = \mathfrak{S}_d$ and $G \not\supset \Gamma(4)$, then the order of $\varpi_8(G)$ is exactly $d! \cdot 2^{4g^2 - 2g}$, and all such $G$ are conjugate as subgroups of $\varpi_2^{-1}(\mathfrak{S}_d)$.

\end{thm}

We will establish some preliminary results needed to prove the above theorem in \S\ref{sec2}.  We will then proceed to give a full proof of Theorem \ref{thm main} by classifying all possible modulo-$4$ images of subgroups $G$ satisfying the hypotheses of the theorem in \S\ref{sec3} (in particular, \S\ref{sec3.3} contains the proof of Theorem \ref{thm main}(a)), classifying all possible modulo-$8$ images in \S\ref{sec4}, and then showing that $G \supset \Gamma(8)$ in \S\ref{sec5}.  Finally, we will discuss an elementary application to even-degree polynomials with full Galois group in \S\ref{sec6}.  Before moving on to \S\ref{sec2}, however, we describe the main motivation for this problem.

\subsection{Application to abelian varieties} \label{S1.2}

The context which led the author to the problem of lifting symmetric groups to subgroups of $\Sp_{2g}(\zz_2)$ is as follows.  Let $K$ be a field of characteristic different from $2$, and let $f(x) \in K[x]$ be a squarefree polynomial of degree $d \geq 3$ with Galois group $\Gal(f)$; we denote the splitting field of $f$ by $L$.  We consider the hyperelliptic curve which is the smooth projective model $C$ of the affine curve given by the equation $y^2 = f(x)$, which has genus $g := \lfloor \frac{d - 1}{2} \rfloor$.  The $2g$-dimensional representation over $\ff_2$ of $\Gal(f) \subseteq S_d$ coming from the standard representation has an interpretation coming from the geometry of $C$, which we briefly summarize below (for further details, see \cite[\S2]{zarhin2002very} and \cite[\S5]{zarhin2001hyperelliptic}).

Associated to the smooth, projective, genus-$g$ curve $C$ is its \textit{$2$-adic Tate module $T_2$}, which is a free $\zz_2$-module of rank $2g$, and which comes endowed with an alternating $\zz_2$-bilinear pairing $e_2 : T_2 \times T_2 \to \zz_2$ called the \textit{Weil pairing} (here we are defining it with respect to the canonical principal polarization on the Jacobian).  The image modulo $2$ of $T_2$ is identified with the even-cardinality subsets of the roots $\alpha_1, ... , \alpha_d \in \bar{K}$ of the polynomial $f$, where subsets $I$ and $I'$ are identified if $I' = \{\alpha_1, ... , \alpha_d\} \smallsetminus I$ in the case that $d$ is even.  We may therefore identify the modulo-$2$ image of $T_2$ with the space $\bar{V}$ defined at the start of \S\ref{sec1.1} in an obvious way, and it is known that the $\ff_2$-bilinear pairing $\langle \cdot, \cdot \rangle$ on $\bar{V}$ described there is the one induced by the modulo-$2$ reduction of $e_2$.  There is a natural representation of the absolute Galois group $\Gal(\bar{K} / K)$ on $T_2$ which respects the pairing $e_2$ up to homothety (in fact, the Galois automorphisms fixing all $2$-power roots of unity in $\bar{K}$ act on $T_2$ as elements of $\Sp(T_2)$).  Its image in $\Aut(T_2)$, which we denote by $G_2$, is closed under the $2$-adic topology.  It is well known that the induced action of $\Gal(\bar{K} / K)$ on $T_2 / 2T_2$ factors through $\Gal(L / K) = \Gal(f) \subseteq S_d$, which acts on $T_2 / 2T_2$ through the standard representation $\rho : S_d \to \Sp(T_2 / 2T_2)$.  In particular, the modulo-$2$ image of $G_2 \cap \Sp(T_2)$ coincides with $\mathfrak{S}_d \subseteq \Sp(T_2 / 2T_2)$ when $\Gal(f) = S_d$.

It would be interesting to know what kind of subgroup $G_2 \cap \Sp(T_2) \subseteq \Sp(T_2)$ may appear as the image of this $2$-adic Galois representation intersected with $\Sp(T_2)$ in the case that our polynomial $f$ has full Galois group $S_d$.  In \cite{zarhin2002very}, Zarhin investigated $\ell$-adic Galois representations associated to hyperelliptic curves over fields $K$ of characteristic different from $2$, defined by equations of the form $y^2 = f(x)$, where $f(x) \in K[x]$ is a squarefree polynomial of degree $d \geq 5$ with large Galois group.  In particular, he showed (see \cite[Theorems 2.4, 2.5, and 2.6]{zarhin2002very}) that under certain hypotheses on the ground field $K$, the natural representation discussed above of the absolute Galois group $\Gal(\bar{K} / K)$ on the associated $2$-adic Tate module $T_2$ contains an open subgroup of $\Sp(T_2)$ provided that $\Gal(f)$ contains a large enough subgroup of $S_d$.  Theorem \ref{thm main} allows us to independently prove a stronger version of this statement in the case that $\Gal(f) = S_d$, as follows.

\begin{cor} \label{cor Zarhin}

Let $K$ be a field of characteristic different from $2$, and let $f(x) \in K[x]$ be a polynomial of degree $d \geq 5$ whose Galois group is the full $S_d$.  Let $\rho_2 : \Gal(\bar{K} / K) \to \Aut(T_2)$ denote the natural action of the absolute Galois group of $K$ on the $2$-adic Tate module $T_2$ of the Jacobian of the hyperelliptic curve over $K$ defined by the equation $y^2 = f(x)$.  Then the image $G_2$ of this representation contains the principal congruence subgroup $\Gamma(2) \subset \Sp(T_2)$ (resp. $\Gamma(8) \lhd \Sp(T_2)$) if $d$ is even (resp. if $d$ is odd); thus, $G_2$ contains an open subset of $\Sp(T_2)$.

\end{cor}

\begin{proof}

As was explained in the above discussion, the action of $\rho_2 : \Gal(\bar{K} / K) \to \Aut(T_2)$ composed with the reduced-modulo-$2$ map $\Sp(T_2) \twoheadrightarrow \Sp(T_2 / 2T_2)$ factors through $\Gal(f) \subseteq S_d$, which acts on the $2g$-dimensional $\ff_2$-vector space $T_2 / 2T_2$ through the standard representation of $S_d$.  Therefore, if $\Gal(f)$ is the full $S_d$, the image modulo $2$ of $G_2$ coincides with $\mathfrak{S}_d \subseteq \Sp(T_2 / 2T_2)$.  Then by Theorem \ref{thm main}, the intersection $G_2 \cap \Sp(T_2)$ contains the full subgroup $\Gamma(2) \lhd \Sp(T_2)$ (resp. $\Gamma(8) \lhd \Sp(T_2)$) if $d$ is even (resp. if $d$ is odd); in particular, $G_2 \cap \Sp(T_2) \subseteq \Sp(T_2)$ is open.

\end{proof}

\begin{rmk}

a) A variant on the above corollary applies to the $2$-adic Galois representation associated to any $g$-dimensional abelian variety over $K$ such that the modulo-$2$ action contains $\mathfrak{S}_{2g + 1}$ or $\mathfrak{S}_{2g + 2}$: this says that the restriction of the $2$-adic image of Galois to the group of symplectic automorphisms of the $2$-adic Tate module is very large, thus providing a strengthening of special cases of \cite[Theorems 2.2 and 2.3]{zarhin2002very}.

b) The results of Zarhin cited above also suggest that the conclusions given by Theorem \ref{thm main} might still hold under the weaker hypotheses that $\varpi_2(G)$ contains the image of the alternating subgroup $A_d$.  Unfortunately, our methods, which heavily involve generating images of $G$ using lifts of transpositions in $\mathfrak{S}_d$, do not seem easily adaptable to providing an answer to this problem.

\end{rmk}

\begin{rmk}

Viewing liftings of images of $\mathfrak{S}_d$ in the context of abelian varieties allows us to see that for $d = 3, 4$, we do not get the conclusion of Theorem \ref{thm main}(b) in general.  Indeed, let $E$ be the elliptic curve defined by the equation $y^2 = f(x) := x^3 - 2$.  Since the cubic $f(x)$ has full Galois group, we get that the image $G \subset \Aut(T_2)$ of the natural $2$-adic Galois action on the $2$-adic Tate module $T_2$ satisfies $\varpi_2(G) = \mathfrak{S}_3$.  However, $E$ is well known to have complex multiplication, implying that $G$ contains an abelian subgroup of finite index in $G$ (see \cite[Corollary 2 in \S4]{serre1968good}), and so $G \cap \Sp(T_2) \subset \Sp(T_2)$ is certainly not open under the $2$-adic topology.

\end{rmk}

\subsection{Acknowledgments} \label{S1.3}

The author would like to thank Aaron Landesman (whose article \cite{landesman2017lifting}, co-authored with A. Swaminathan, J. Tao, and Y. Xu, helped to inspire this work) for discussions and computations which were helpful in approaching this problem.  The author is also very grateful to Dan Collins for providing computational evidence of this paper's modulo-$8$ and modulo-$16$ results in the key case that $g = 2$ and $d = 5$; this was instrumental in guiding the author towards the correct conclusions.

\section{Preliminaries} \label{sec2}

We retain all of the notation from \S\ref{sec1}; in particular, we let $V$ be a free $\zz_2$-module of rank $2g$ equipped with a symplectic form $\langle \cdot, \cdot \rangle$ and let $G \subseteq \Sp(V)$ be a closed subgroup.  Given any element $a \in V$, we define the \textit{transvection} with respect to $a$ to be the linear automorphism $t_a \in \Sp(V)$ given by $v \mapsto v + \langle v, a \rangle a$.  We define transvections similarly over the quotients $V / 2^n V$, using the same notation $t_a \in \Sp(V / 2^n V)$ for $a \in V / 2^n V$.  When $a = v_I \in V / 2V$ corresponds to an even-cardinality subset $I \subseteq \{1, ... , d\}$, we simply write $t_I$ for $t_a$.  Below we will freely use the easily-verifiable fact that the image in $\Sp(V / 2V)$ of any transposition $(i, j) \in S_d$ is the transvection $t_{\{i, j\}}$.  (In several places below we mildly abuse notation by expressing elements of $\End(V / 2^{n + 1}V)$ for some $n \geq 1$ which are divisible by $2^n$ as $2^n$ times some element of $\End(V / 2V)$.)

We now fix, for the rest of this paper, a choice of vectors $a_{i, j} = a_{j, i} \in V$ for $1 \leq i < j \leq 2g + 1$ which satisfy $\varpi_2(a_{i, j}) = v_{\{i, j\}}$.

In order to prove Theorem \ref{thm main}, it clearly suffices to prove the statement under the assumption that $\varpi_2(G) = \mathfrak{S}_d$ for $d = 2g + 1$ or $d = 2g + 2$.  We therefore assume that $\varpi_2(G) = \mathfrak{S}_d$ for some fixed $d \in \{2g + 1, 2g + 2\}$.  To avoid cumbersome arguments, we also assume for the rest of the paper that $d \neq 4$, so that the symmetric group $S_d$ may be identified with its image $\mathfrak{S}_d \in \Sp(V / 2V)$ (when $d = 4$, we are simply reduced to the $d = 3$ case anyway because we have $\mathfrak{S}_4 = \mathfrak{S}_3$).  We begin by looking more closely at the subgroups $\varpi_{2^{n + 1}}(\Gamma(2^n)) \lhd \Sp(V / 2^{n + 1}V)$ for $n \geq 1$.

\subsection{The groups $\varpi_{2^{n + 1}}(\Gamma(2^n))$} \label{sec2.1}

We now recall the following description of $\varpi_{2^{n + 1}}(\Gamma(2^n))$ for $n \geq 1$.  Let $\mathfrak{sp}(V / 2V)$ denote the $\ff_2$-vector space of endomorphisms of $V / 2V$ generated by the elements of the form $T_a := t_a - 1$ for all $a \in V$ (the group variety $\mathfrak{sp}$ gives the symplectic Lie algebra).  It is well known (see \cite[\S A.3]{kirillov2008introduction}) that $\mathfrak{sp}(V / 2V)$ has dimension $2g^2 + g$ and that the map $\mathfrak{sp}(V / 2V) \to \varpi_{2^{n + 1}}(\Gamma(2^n))$ given by $X \mapsto 1 + 2^n X$ is an isomorphism.  Therefore, the multiplicative group $\varpi_{2^{n + 1}}(\Gamma(2^n))$ is an elementary abelian $2$-group of rank $2g^2 + g$.  More precisely, we have the following.

\begin{prop} \label{prop 2-group}

Fix an integer $n \geq 1$.

a) A basis for the elementary abelian $2$-group $\varpi_{2^{n + 1}}(\Gamma(2^n))$ is given by $\{\varpi_{2^{n + 1}}(t_{a_{i, j}}^{2^n})\}_{1 \leq i < j \leq 2g + 1}$.

b) Let $u \in \varpi_{2^{n + 1}}(G)$ be any element, and let $\sigma = \varpi_{2^{n + 1} \to 2}(u) \in \mathfrak{S}_d$ be its image viewed as a permutation in $S_d$.  Then conjugation by $u$ permutes the elements $\varpi_{2^{n + 1}}(t_{a_{i, j}}^{2^n}) \in \varpi_{2^{n + 1}}(\Gamma(2^n))$ by the formula $u \varpi_{2^{n + 1}}(t_{a_{i, j}}^{2^n}) u^{-1} = \varpi_{2^{n + 1}}(t_{a_{\sigma(i), \sigma(j)}}^{2^n})$.

\end{prop}

\begin{proof}

Note that $\varpi_{2^{n + 1}}(t_{a_{i, j}}^{2^n}) = 1 + 2^n T_{\{i, j\}}$ regardless of our choice of $a_{i, j}$.  Therefore, by the isomorphism $\mathfrak{sp}(V / 2V) \stackrel{\sim}{\to} \varpi_{2^{n + 1}}(\Gamma(2^n))$ given in the above discussion, in order to prove part (a) it suffices to show that the elements $T_{\{i, j\}}$ for $1 \leq i < j \leq 2g + 1$ form a basis for $\mathfrak{sp}(V / 2V)$.  Since the cardinality of this set is $2g^2 + g$, which is the dimension of $\mathfrak{sp}(V / 2V)$, we only need to show that the set spans $\mathfrak{sp}(V / 2V)$.  This is already given by \cite[Lemma 2.2.1]{brumer2018large}, but it is convenient to provide a different proof here.  Let $v_I \in V / 2V$ be any vector, where $I \subseteq \{1, ... , d\}$ is an even-cardinality subset.  If $d$ is even and $2g + 2 \in I$, we may replace $I$ with its complement in $\{1, ... , d\}$ without changing $v_I$; we therefore assume that $I \subset \{1, ... , 2g + 1\}$.  Lemma \ref{lemma compliment} below then says that $T_I$ can be expressed as a sum of elements in $\{T_{\{i, j\}}\}_{1 \leq i < j \leq 2g + 1}$.  Since $\mathfrak{sp}(V / 2V)$ is generated by such elements, this set does span $\mathfrak{sp}(V / 2V)$, as desired.

To prove part (b), we first observe that we have $u \varpi_{2^{n + 1}}(t^{2^n}_{a_{i, j}}) u^{-1} = t_{u(\varpi_{2^{n + 1}}(a_{i, j}))}^{2^n}$.  Since the modulo-$2$ image of $u(\varpi_{2^{n + 1}}(a_{i, j}))$ is clearly $\sigma(v_{\{i, j\}}) = v_{\{\sigma(i), \sigma(j)\}}$, we have $1 + 2^n T_{\{\sigma(i), \sigma(j)\}} = t_{u(\varpi_{2^{n + 1}}(a_{i, j}))}^{2^n}$, whence the statement.

\end{proof}

\begin{lemma} \label{lemma compliment}

For any even-cardinality subset $I \subseteq \{1, ... , d\}$, we have the identity 
\begin{equation} \label{eq identity1}
T_I = \sum_{i, j \in I; \ i < j} T_{\{i, j\}}.
\end{equation}

\end{lemma}

\begin{proof}

This can be verified directly using the formula $T_{\{i, j\}} : v_J \mapsto \#(J \cap \{i, j\})v_{\{i, j\}}$ for elements $J \in V / 2V$ corresponding to $2$-element subsets of $\{1, ... , d\}$ (which in turn clearly generate $V / 2V$).

\end{proof}

We can now show that we may reduce our problem to a study of certain subgroups of $\varpi_{4 \to 2}^{-1}(\mathfrak{S}_d)$.

\begin{lemma} \label{lemma sufficient}

If $H \subseteq \Sp(V)$ is a closed subgroup satisfying $\varpi_{2^{n + 1}}(H) \supseteq \varpi_{2^{n + 1}}(\Gamma(2^n))$ for some $n \geq 1$, then we have $H \supseteq \Gamma(2^n)$.  Therefore, in order to prove Theorem \ref{thm main}, it suffices to show that if $d$ is even (resp. if $d$ is odd), we have $\varpi_4(G) \supset \varpi_4(\Gamma(2))$ (resp. $\varpi_{16}(G) \supset \varpi_{16}(\Gamma(8))$), and that in the $d$ odd case, $\varpi_8(G) \not\supset \varpi_8(\Gamma(4))$ implies that $\varpi_8(G)$ has order $d! \cdot 2^{4g^2 - 2g}$.

\end{lemma}

\begin{proof}

Since $H$ is closed, it is enough to show that $\varpi_{2^{m + 1}}(H) \supseteq \varpi_{2^{m + 1}}(\Gamma(2^m))$ for all $m \geq n$.  This statement is true for $m = n$ by hypothesis; assume that it holds for $m - 1 \geq n$.  Then the group $\varpi_{2^m}(H)$ contains the elements $1 + 2^{m - 1} T_{a_{i, j}} \in \Sp(V / 2^m V)$ for $1 \leq i < j \leq 2g + 1$.  These each lift to elements $1 + 2^{m - 1} T_{a_{i, j}} + 2^m B_{i, j} \in \varpi_{2^{m + 1}}(H)$ for some $B_{i, j} \in \End(V / 2V)$.  The squares of these elements are $1 + 2^m T_{a_{i, j}} \in \varpi_{2^{m + 1}}(H)$, which by Proposition \ref{prop 2-group}(a) generate $\varpi_{2^{m + 1}}(\Gamma(2^m))$, so we have proved the statement by induction.  (See also the proofs of \cite[Lemma 4]{cullinan2017jacobians} and \cite[Lemma 6]{landesman2017lifting}.)

\end{proof}

In light of Proposition \ref{prop 2-group}, given a fixed $n \geq 1$, for $1 \leq i < j \leq d$, we write $[i, j] = [j, i]$ for the element $\varpi_{2^{n + 1}}(t_{a_{i, j}}^{2^n}) \in \varpi_{2^{n + 1}}(\Gamma(2^n))$ (noting that it does not depend on our choice of $a_{i, j}$'s).  Proposition \ref{prop 2-group}(b) says that $S_d$ acts on $\varpi_{2^{n + 1}}(\Gamma(2^n))$ by permuting this set of generators as $^{\sigma}[i, j] = [\sigma(i), \sigma(j)]$.  Since $\varpi_{2^{n + 1}}(\Gamma(2^n))$ is elementary abelian, we treat it as a vector space over $\ff_2$ and use additive notation when expressing its elements in terms of the $[i, j]$'s.  For distinct $i, j, k \in \{1, ... , d\}$, we write $\Delta_{i, j, k}$ to denote the element $[i, j] + [i, k] + [j, k]$.  (In order to avoid excessive clutter, we suppress any indication of $n$ in our notation for the elements $[i, j]$ and $\Delta_{i, j, k}$ but always make sure it is clear which space $\varpi_{2^{n + 1}}(\Gamma(2^n))$ they belong to.)

We remark that each $\ff_2$-space $\varpi_{2^{n + 1}}(\Gamma(2^n))$ has a basis given by $\{[i, j]\}_{1 \leq i < j \leq 2g + 1}$ by Proposition \ref{prop 2-group}(a) and therefore has a (unique) $S_d$-invariant subspace of dimension $1$, namely the subspace consisting of all elements $\sum_{1 \leq i < j \leq 2g + 1} c_{i, j}[i, j] \in \varpi_{2^{n + 1}}(\Gamma(2^n))$ with $\sum_{1 \leq i < j \leq 2g + 1} c_{i, j} = 0$.  We denote this subspace by $\varpi_{2^{n + 1}}(\Gamma(2^n))_0$.

\subsection{The subspaces $N^{(2^n)} \subset \varpi_{2^{n + 1}}(\Gamma(2^n))$} \label{sec2.2}

We now define an $S_d$-invariant subspace of each $\varpi_{2^{n + 1}}(\Gamma(2^n)) = \bigoplus_{i < j} \ff_2 [i, j]$ which will turn out to be very crucial.

\begin{dfn}

Fix an integer $n \geq 1$.  We define $N^{(2^n)} \subset \varpi_{2^{n + 1}}(\Gamma(2^n))$ to be the subset consisting of all elements $\sum_{1 \leq i < j \leq 2g + 1} c_{i, j} [i, j]$ (with $c_{i, j} = c_{j, i} \in \ff_2$) satisfying the following property: for each $i \in \{1, ... , 2g + 1\}$, we have $\sum_{j \in \{1, ... , 2g + 1\} \smallsetminus \{i\}} c_{i, j} = 0$.

\end{dfn}

It is easy to see that $N^{(2^n)}$ is a subspace of $\varpi_{2^{n + 1}}(\Gamma(2^n))$; we now present a proposition describing some of its properties.

\begin{prop} \label{prop N}

For each $n \geq 1$, the subspace $N^{(2^n)} \subset \varpi_{2^{n + 1}}(\Gamma(2^n))$ satisfies the following.

a) The subspace $N^{(2^n)} \subset \varpi_{2^{n + 1}}(\Gamma(2^n))$ is generated by the elements $\Delta_{i, j, k}$ defined above over all distinct $i, j, k \in \{1, ... , 2g + 1\}$.  A full set of linear relations amongst these generators is given by $\Delta_{i, j, k} + \Delta_{i, j, l} + \Delta_{i, k, l} + \Delta_{j, k, l} = 0$ for all distinct $i, j, k, l \in \{1, ... , 2g + 1\}$.

b) The subspace $N^{(2^n)} \subset \varpi_{2^{n + 1}}(\Gamma(2^n))$ is $S_d$-invariant.  Moreover, it is maximal among proper $S_d$-invariant subspaces of $\varpi_{2^{n + 1}}(\Gamma(2^n))$.

c) The intersection $N^{(2^n)}_0 := N^{(2^n)} \cap \varpi_{2^{n + 1}}(\Gamma(2^n))_0$ is a codimension-$1$ subspace of $N^{(2^n)}$ satisfying the following property: if $W$ is an $S_d$-invariant subspace of $\varpi_{2^{n + 1}}(\Gamma(2^n))$ with $N^{(2^n)}_0 \subset W \not\subset \varpi_{2^{n + 1}}(\Gamma(2^n))_0$, then $W$ contains $N^{(2^n)}$ (and therefore coincides with $N^{(2^n)}$ or with $\varpi_{2^{n + 1}}(\Gamma(2^n))$).

d) The subspace $N^{(2^n)} \subset \varpi_{2^{n + 1}}(\Gamma(2^n))$ has dimension $2g^2 - g$, and the quotient space $M^{(2^n)} := \varpi_{2^{n + 1}}(\Gamma(2^n)) / N^{(2^n)}$ has dimension $2g$.  There is an isomorphism $V / 2V \stackrel{\sim}{\to} M^{(2^n)}$ given by sending $v_{\{i, j\}} \in V / 2V$ to the image modulo $N^{(2^n)}$ of $[i, j]$.  The action of $S_d$ on $V / 2V$ induced by the composition of surjections $\varpi_{2^{n + 1}}(\Gamma(2^n)) \twoheadrightarrow M^{(2^n)} \stackrel{\sim}{\to} V / 2V$ is the standard representation $\rho$.

e) For $1 \leq s \leq n - 1$, we have that $\varpi_{2^{n + 1}}(\Gamma(2^s))$ and $\varpi_{2^{n + 1}}(\Gamma(2^{n + 1 - s}))$ commute and that the commutator of $\varpi_{2^{n + 1}}(\Gamma(2^s))$ and $\varpi_{2^{n + 1}}(2^{n - s}))$ coincides with $N^{(2^n)}$.  More explicitly, the commutator of $\varpi_{2^{n + 1}}(t_{a_{i, j}})^{2^s}$ and $\varpi_{2^{n + 1}}(t_{a_{i, k}})^{2^{n - s}}$ is given by $\Delta_{i, j, k} \in N^{(2^n)}$.

f) For $0 \leq s \leq n$, given vectors $a, b \in V / 2^{n + 1}V$ with $a \equiv b$ (mod $2^s$), we have $t_b^{2^{n - s}} t_a^{-2^{n - s}} \in N^{(2^n)}$.

\end{prop}

\begin{proof}

Choose any element $u \in \varpi_{2^{n + 1}}(\Gamma(2^n))$, whose basis expansion is $\sum_{1 \leq i < j \leq 2g + 1} c_{i, j} [i, j]$ for some scalars $c_{i, j}$.  For each pair of distinct $i, j \geq 2$ such that $c_{i, j} = 1$, we add the element $\Delta_{1, j, k} \in N^{(2^n)}$ to $u$ until the only basis elements left appearing in the linear expansion are of the form $[1, i]$; we write $u'$ for this resulting vector.  It follows directly from the definition of $N^{(2^n)}$ that such an element $u'$ lies in $N^{(2^n)}$ if and only if it is trivial, which implies that $u \in N^{(2^n)}$ if and only if $u$ is the sum of elements of the form $\Delta_{i, j, k}$.  This proves that $N^{(2^n)}$ is generated by the $\Delta_{i, j, k}$'s, which is the first statement of part (a).

Now choose any element $u = \sum_{i, j, k} m_{i, j, k}\Delta_{i, j, k} \in N^{(2^n)}$ for scalars $m_{i, j, k} \in \ff_2$.  If we consider $u$ as an element of the vector space $\varpi_{2^n}(\Gamma(2^{n + 1}))$ and write the basis expansion $\sum_{1 \leq i < j \leq 2g + 1} c_{i, j} [i, j]$ for coefficients $c_{i, j} \in \ff_2$, we have $c_{i, j} = \sum_{k \neq i, j} m_{i, j, k}$ for each $i, j$.  We assume that $u = 0$, so we have that $\sum_{k \neq i, j} m_{i, j, k} = 0$ for each $i, j$.  If all scalars $m_{i, j, k}$ are trivial, we are done, so we assume without loss of generality that the (even-cardinality) subset $I \subset \{3, ... , 2g + 1\}$ whose elements $k$ satisfy $m_{1, 2, k} = 1$ is nonempty.  We partition $I$ into $2$-element subsets, and for each such subset $\{k, l\}$ subtract the element $\Delta_{1, 2, k} + \Delta_{1, 2, l} + \Delta_{1, k, l} + \Delta_{2, k, l}$ from the sum $\sum_{i, j, k} m_{i, j, k}\Delta_{i, j, k}$; the resulting sum $\sum_{i, j, k} m_{i, j, k}' \Delta_{i, j, k}$ satisfies $m_{1, 2, k} = 0$ for $3 \leq k \leq 2g + 1$.  Now for each other pair $i, j$ with $1 \leq i < j \leq 2g + 1$ such that $m_{i, j, k} = 1$ for some $k \neq i, j$, we can repeat this process of getting rid of all summands $\Delta_{i, j, k}$ by subtracting sums of the form $\Delta_{i, j, k} + \Delta_{i, j, l} + \Delta_{i, k, l} + \Delta_{j, k, l}$ until the sum is trivial.  This proves that such expressions provide a full set of relations among the $\Delta_{i, j, k}$'s and gives the second statement of part (a).

Now to prove that $N^{(2^n)}$ is $S_d$-invariant, it is enough to check how elements of $S_d$ act on the generators given by part (a).  It is clear from the definition of the elements $\Delta_{i, j, k}$ that $S_d$ permutes the set $\{\Delta_{i, j, k}\}_{i, j, k \in \{1, ... , d\}}$, so it suffices to show that all elements in this set lie in $N^{(2^n)}$.  When $i, j, k \neq 2g + 2$, we have $\Delta_{i, j, k} \in N^{(2^n)}$ by definition.  In the case that $d = 2g + 2$, given any distinct $i, j \in \{1, ... , 2g + 1\}$, we apply the formula in Lemma \ref{lemma compliment} to get 
\begin{equation}
\Delta_{i, j, 2g + 2} = \sum_{k \in \{1, ... , 2g + 1\} \smallsetminus \{i, j\}} \Delta_{i, j, k}.
\end{equation}
  Therefore, we have $\Delta_{i, j, k} \in N^{(2^n)}$ for all distinct $i, j, k \in \{1, ... , d\}$, thus proving the first statement of part (b).

Now let $W \subseteq \varpi_{2^{n + 1}}(\Gamma(2^n))$ be an $S_d$-invariant subspace which properly contains $N^{(2^n)}$, and choose a vector $w \in W \smallsetminus N^{(2^n)}$.  We shall show that this assumption implies that $W = \varpi_{2^{n + 1}}(\Gamma(2^n))$.  By what we have shown to prove part (a), we may write $w$ as a sum of an element in $N^{(2^n)}$ and the element $w' := \sum_{j \in I} [1, j]$ for some subset $I \subseteq \{2, ... , 2g + 1\}$.  Since $w \notin N^{(2^n)}$, we have $w' \neq 0$ and so $I \neq \varnothing$.  If $I = \{2, ... , 2g + 1\}$, let $\sigma = (1, 2) \in S_d$; otherwise, let $\sigma = (i, j) \in S_d$ for some $i \in I$ and $j \in \{2, ... , 2g + 1\} \smallsetminus I$.  Then by $S_d$-invariance we get $w' + ^{\sigma}\!w' \in W$ and, after subtracting an appropriate element of $N^{(2^n)}_0 \subset W$, we either get that $[1, 3] + [2, 3] \in W$ (if $I = \{2, ... , 2g + 1\}$) or we get that $[1, i] + [1, j] \in W$ (otherwise).  So by $S_d$-invariance, we have $[i, j] + [i, k] \in W$ for all distinct $i, j, k \in \{1, ... , 2g + 1\}$, and it is easy to see that these elements generate $\varpi_{2^{n + 1}}(\Gamma(2^n))_0$.  But since the subspace $W$ contains the element $[1, 2] + [1, 3] + [2, 3] \in N^{(2^n)}$, we have $W \supsetneq \varpi_{2^{n + 1}}(\Gamma(2^n))_0$ and the desired conclusion results from the maximality of $\varpi_{2^{n + 1}}(\Gamma(2^n))_0$ as a subspace of $\varpi_{2^{n + 1}}(\Gamma(2^n))$.  We have thus proved part (b).

Now let $W \subseteq \varpi_{2^{n + 1}}(\Gamma(2^n))$ be an $S_d$-invariant subspace which contains $N^{(2^n)}_0$ and is not contained in $\varpi_{2^{n + 1}}(\Gamma(2^n))_0$, and choose a vector $w \in W \smallsetminus N^{(2^n)}_0$.  We observe that $\Delta_{1, 2, 3} \in N^{(2^n)} \smallsetminus N^{(2^n)}_0$ and so $N^{(2^n)}_0 \subsetneq N^{(2^n)}$ is a subspace of codimension $1$.  Therefore, if $w \in N^{(2^n)}$ we have $W \supseteq N^{(2^n)}$.  In the case that $w \notin N^{(2^n)}$, then we see that that $W \supseteq \varpi_{2^{n + 1}}(\Gamma(2^n))_0$ by the exact same argument as was used to prove the second statement of part (b).  Now, since $W$ by hypothesis is not contained in $\varpi_{2^{n + 1}}(\Gamma(2^n))_0$, we get $W = \varpi_{2^{n + 1}}(\Gamma(2^n))$, and part (c) is proved.

We established above that the $2^{2g}$-element set of elements of the form $\sum_{i \in I} [1, i]$ for all subsets $I \subseteq \{2, ... , 2g + 1\}$ is a set of coset representatives for $N^{(2^n)} \subset \varpi_{2^{n + 1}}(\Gamma(2^n))$.  Therefore, the induced quotient $M^{(2^n)}$ has dimension $2g$, implying that $N^{(2^n)}$ has dimension $(2g^2 + g) - 2g = 2g^2 - g$.  It is now straightforward to check that the map $V / 2V \to M^{(2^n)}$ given in the statement of (c) is an isomorphism, using the fact that $[i, j] + [i, k] \equiv [j, k]$ (mod $N^{(2^n)}$) for distinct $i, j, k$.  Finally, the fact that the induced representation on $V / 2V$ is simply the standard representation $\rho$ follows easily from the fact that $^{\sigma}[i, j] = [\sigma(i), \sigma(j)]$ for $1 \leq i < j \leq d$, proving part (d).

One proves (e) by first checking directly that modulo $2^{n + 1}$, we have the equivalence 
\begin{equation}
(1 + 2^s T_{\{i, j\}})(1 + 2^{n - s}T_{\{i, k\}}) \equiv (1 + 2^{n - s} T_{\{i, k\}})(1 + 2^sT_{\{i, j\}})(1 + 2^n T_{\{i, j\}})(1 + 2^n T_{\{i, k\}})(1 + 2^n T_{\{j, k\}})
\end{equation}
 for any distinct $i, j, k$, where the elements $T_{\{i, j\}} \in \mathfrak{sp}(V)$ are defined as in the above discussion.  Since we have the easily-derived identity $t_{\{i, j\}}^{2^r} = 1 + 2^r T_{\{i, j\}}$ for integers $r \geq 0$, this gives the formula for the commutator given in the statement.  It also implies that $\varpi_{2^n}(\Gamma(2^s))$ and $\varpi_{2^n}(\Gamma(2^{n - s}))$ commute in $\Sp(V / 2^n V)$, which is the first statement of (e) (with $n$ replaced by $n - 1$).  This statement, combined with the fact that each $\varpi_{2^{n + 1}}(\Gamma(2^s))$ is generated by $\varpi_{2^{n + 1}}(\Gamma(2^{s + 1}))$ and the elements $\varpi_{2^{n + 1}}(t_{a_{i, j}})^s$ by Proposition \ref{prop 2-group}(a) along with part (a) of the proposition we are proving, implies that the commutator of $\varpi_{2^{n + 1}}(\Gamma(2^s))$ and $\varpi_{2^{n + 1}}(2^{n - s}))$ coincides with $N^{(2^n)}$.

Finally, to prove (f), we first note that we must have $b = t(a)$ for some $t \in \varpi_{2^{n + 1}}(\Gamma(2^s))$, and since $\varpi_{2^{n + 1}}(\Gamma(2^s))$ is generated by elements of the form $t_c^{2^s}$, we may assume that $b = t_c^{2^s}(a)$ for some $c \in V / 2^{n + 1}V$.  Then we get $t_b^{2^{n - s}} t_a^{-2^{n - s}} = t_c^{2^s} t_a^{2^{n - s}} t_c^{-2^s} t_a^{-2^{n - s}}$, which by part (e) lies in $N^{(2^n)}$.

\end{proof}

\begin{rmk} \label{rmk N}

Part (a) of the proposition above effectively provides an alternate definition of $N^{(2^n)}$; here we present a third definition which is easily seen to be equivalent.  Let $\langle \cdot, \cdot \rangle : \varpi_4(\Gamma(2)) \times \varpi_4(\Gamma(2)) \to \ff_2$ be the alternating bilinear pairing given by setting $\langle [i, j], [k, l] \rangle = \#(\{i, j\} \cap \{k, l\})$ and extending $\ff_2$-linearly.  Then $N^{(2^n)}$ coincides with the kernel of this pairing, and the induced pairing on $M^{(2^n)}$ is the one given by transport of structure via the isomorphism $V / 2V \stackrel{\sim}{\to} M^{(2^n)}$ described in part (c).

\end{rmk}

\section{Lifting to $\zz / 4\zz$} \label{sec3}

We will now find all possible liftings of $\mathfrak{S}_d$ to subgroups of $\Sp(V / 4V)$, thus classifying all possible subgroups $\varpi_4(G) \subset \varpi_{4 \to 2}^{-1}(\mathfrak{S}_d)$.  We retain all previous notation, in particular our choice of lifts $a_{i, j} \in V$ of the vectors $v_{\{i, j\}} \in V / 2V$ and the subgroup $N^{(2)} \subset \varpi_4(\Gamma(2))$ defined in \S\ref{sec2.2} with its corresponding quotient $M^{(2)}$.  Since it should not cause confusion, throughout this section we denote the subspaces $N^{(2)}$ and $M^{(2)}$ simply by $N$ and $M$ respectively.

\subsection{The possible images of $G \cap \Gamma(2)$ modulo $4$} \label{sec3.1}

We now show that the only possibilities for the subgroup $\varpi_4(G \cap \Gamma(2)) \subseteq \varpi_4(\Gamma(2))$ are that it is the full $\varpi_4(\Gamma(2))$ or that it coincides with $N$.

\begin{lemma} \label{lemma contains N}

Suppose that $H \subset \varpi_{4 \to 2}^{-1}(\mathfrak{S}_d)$ is a subgroup satisfying $\varpi_{4 \to 2}(H) = \mathfrak{S}_d$.  We have the (strict) containment $N \subset H$.  In fact, we have that $H \not\supset \varpi_4(\Gamma(2))$ implies $H \cap \varpi_4(\Gamma(2)) = N$.

\end{lemma}

\begin{proof}

We start by claiming that $H \cap \varpi_4(\Gamma(2))$ is an $S_d$-invariant subspace of $\varpi_4(\Gamma(2))$ in the sense of \S\ref{sec2.1}.  Given any permutation $\sigma \in S_d$, we can find an element $u \in H$ with $\varpi_{4 \to 2}(u) = \sigma \in \mathfrak{S}_d$ since $H$ surjects onto $\mathfrak{S}_d$ by hypothesis.  Now the action of $\sigma$ on $\varpi_4(\Gamma(2))$ as defined using Proposition \ref{prop 2-group}(b) is simply conjugation by $u \in H$, which clearly stabilizes $H \cap \varpi_4(\Gamma(2)) \lhd H$, thus implying the claim.  Now Proposition \ref{prop N}(b) says that the only $S_d$-invariant subgroup of $\varpi_4(\Gamma(2))$ strictly containing $N$ is $\varpi_4(\Gamma(2))$ itself; therefore, the second statement of the proposition follows from the first.

Choose any transposition in $S_d$, say $\sigma := (1, 2) \in S_d = \mathfrak{S}_d$, and let $u \in H$ be an element with $\varpi_{4 \to 2}(u) = \sigma$.  Choose a vector $a \in V / 4V$ whose modulo-$2$ image is given by $v_{\{1, 2\}} \in V / 2V$.  Then it is straightforward to check that $\varpi_{4 \to 2}(t_a) = \sigma$ so we have $u = \phi t_a$ for some $\phi \in \varpi_4(\Gamma(2))$.  Note that we have $t_a^2 = [1, 2] \in \varpi_4(\Gamma(2))$.  By Proposition \ref{prop 2-group}(a), we have $\phi = \sum_{1 \leq i < j \leq 2g + 1} c_{i, j} [i, j]$ for unique scalars $c_{i, j} \in \ff_2$.  Letting $I \subseteq \{3, ... , 2g + 1\}$ be the subset consisting of all $j$ such that $c_{1, j} \neq c_{2, j}$, we compute using Proposition \ref{prop 2-group}(b) that 
\begin{align}
\begin{split}
\mathfrak{u} := u^2 = \phi (t_a \phi t_a^{-1}) t_a^2 = t_a^2 \phi (t_a \phi t_a^{-1}) &= [1, 2] + \sum_{1 \leq i < j \leq 2g + 1} c_{i, j} ([i, j] + [\sigma(i), \sigma(j)]) \\
 &= [1, 2] + \sum_{3 \leq j \leq 2g + 1} (c_{1, j} + c_{2, j})([1, j] + [2, j]) \\
 &= [1, 2] + \sum_{j \in I} ([1, j] + [2, j]).
\end{split}
\end{align}
  If $I = \varnothing$, then the computation above shows that $\mathfrak{u} = [1, 2]$.  It then follows from the $S_d$-invariance of $H \cap \varpi_4(\Gamma(2))$ and Proposition \ref{prop 2-group}(b) that $[i, j] \in H$ for all $1 \leq i < j \leq 2g + 1$.  Since these elements generate $\varpi_4(\Gamma(2))$ by Proposition \ref{prop 2-group}(a), we are done.  We therefore assume that $I \neq \varnothing$ and choose a natural number $k \in I$.  It is straightforward to compute that 
\begin{equation}
\mathfrak{u} + ^{(2, k)}\!\!\mathfrak{u} + ^{(1, 2, k)}\!\!\mathfrak{u} = [1, 2] + [1, k] + [2, k] \in N.
\end{equation}
Since $\mathfrak{u} + ^{(2, k)}\!\mathfrak{u} + ^{(1, 2, k)}\!\mathfrak{u} \in H \cap \varpi_4(\Gamma(2))$, similarly by $S_d$-invariance we get that the elements $[i, j]+ [i, k] + [j, k] = \Delta_{i, j, k}$ lie in $H$ for all distinct $i, j, k$.  Since these elements generate $N$ by Proposition \ref{prop N}(a), we get $N \subset H$.
\end{proof}

\begin{cor} \label{cor contains N}

Suppose that $H \subset \varpi_{4 \to 2}^{-1}(\mathfrak{S}_d)$ is a subgroup satisfying $\varpi_{4 \to 2}(H) = \mathfrak{S}_d$ and $H \cap \varpi_4(\Gamma(2)) = N$.  For each $\sigma \in \mathfrak{S}_d$, let $u_{\sigma} \in H$ be an element such that $\varpi_{4 \to 2}(u_{\sigma}) = \sigma$.  Then the set of $u_{\sigma}$'s generates $H$.

\end{cor}

\begin{proof}

Let $H' \subseteq H$ be the subgroup generated by the $u_{\sigma}$'s.  Choose any element $u \in H$, and let $\sigma = \varpi_2(u)$.  Then we have $u u_{\sigma}^{-1} \in H \cap \varpi_4(\Gamma(2)) = N$.  But by applying Lemma \ref{lemma contains N} to $H'$, we see that $N \subset H'$ and so $u = (u u_{\sigma}^{-1}) u_{\sigma} \in H'$.  We thus get the desired equality $H' = H$.

\end{proof}

\subsection{Quasi-cocycles} \label{sec3.2}

We now fix, once and for all, a set of lifts $\tilde{\sigma} \in \Sp(V)$ for all $\sigma \in \mathfrak{S}_d$ which satisfies the following hypothesis.

\begin{hyp} \label{hyp lifts}

We assume the following.

a) We have $\tilde{(1)} = 1$ and $\varpi_2(\tilde{\sigma}) = \sigma$ for $\sigma \in \mathfrak{S}_d$.

b) For each transposition $(i, j) \in \mathfrak{S}_d$, we have $\widetilde{(i, j)} = t_{a_{i, j}}$.

\end{hyp}

In practice, we shall only be concerned with lifts of transpositions, and lifts of nontrivial non-transpositions can be chosen arbitrarily.

For the following lemma and below, we write $\pi : \varpi_4(\Gamma(2)) \twoheadrightarrow M$ for the quotient map with respect to $N$.

\begin{lemma} \label{lemma lift identities}

The lifts of transpositions fixed above satisfy the relations 
\begin{equation} \label{eq i j k}
\varpi_8(\widetilde{(i, j)}\widetilde{(i, k)}\widetilde{(i, j)}) = \varpi_8(\widetilde{(i, k)}\widetilde{(i, j)}\widetilde{(i, k)})
\end{equation}
 and 
\begin{equation} \label{eq i j k l}
\pi \circ \varpi_4(\widetilde{(i, j)}\widetilde{(k, l)}) = \pi \circ \varpi_4(\widetilde{(k, l)}\widetilde{(i, j)})
\end{equation}
 for distinct $i, j, k, l$.

\end{lemma}

\begin{proof}

%We recall the formulas $v \mapsto v \pm \langle v, a_{i, j} \rangle a_{i, j}$ and $v \mapsto v \pm \langle v, a_{i, k} \rangle a_{i, k}$ which respectively define the transvections (and its inverse) $t_{a_{i, j}}^{\pm 1}$ and $t_{a_{i, k}}^{\pm 1}$.
To verify the relation (\ref{eq i j k}) holds, we apply the formula $v \mapsto v \pm \langle v, w \rangle w$ which defines the operators $t_w^{\pm 1}$ to $w = a_{i, j}, a_{i, k}$ and compose these formulas appropriately to get a formula for $t_{a_{i, j}}t_{a_{i, k}}t_{a_{i, j}}$.  Using the fact that $\langle a_{i, j}, a_{i, k} \rangle^2 \equiv 1$ (mod $8$), one checks that this formula reduces modulo $8$ to 
\begin{equation}
t_{a_{i, j}} t_{a_{i, j}} t_{a_{i, j}}(v) \equiv v + \langle v, a_{i, j} \rangle a_{i, j} + \langle v, a_{i, k} \rangle a_{i, k} + \langle v, a_{i, j} \rangle \langle a_{i, j}, a_{i, k} \rangle a_{i, k} + \langle v, a_{i, k} \rangle \langle a_{i, k}, a_{i, j} \rangle a_{i, j}.
\end{equation}
  The desired relation can now be seen from the fact that $a_{i, j}$ and $a_{i, k}$ are interchangeable in the above expression.

It is clear that $\langle a_{i, j}, a_{k, l} \rangle$ is equivalent modulo $4$ to either $2$ or $0$.  To verify the relation (\ref{eq i j k l}) in the former case, one checks by composing formulas for the transevections (and their inverses) $t_{a_{i, j}}^{\pm 1}$ and $t_{a_{i, k}}^{\pm 1}$ as above that these operators commute modulo $4$.  For the latter case, one composes the same formulas defining transvections and their inverses in order to get a formula for the commutator $t_{a_{i, j}} t_{a_{k, l}}t_{a_{i, j}}^{-1}t_{a_{k, l}}^{-1}$ and show that its reduction modulo $4$ is the same as that of the formula for $t_{a_{i, j} + a_{k, l}}^2 t_{a_{i, j}}^2 t_{a_{k, l}}^2$.  Meanwhile, using Lemma \ref{lemma compliment} we compute that the image of $t_{a_{i, j} + a_{k, l}}^2 t_{a_{i, j}}^2 t_{a_{k, l}}^2$ in $\varpi_4(\Gamma(2))$ is equal to $\Delta_{i, j, k} + \Delta_{j, k, l} \in N$, and thus it lies in the kernel of $\pi \circ \varpi_4$, as desired.

\end{proof}

\begin{dfn} \label{dfn quasi-cocycle}

A (level-$4$) \textit{quasi-cocycle} is a map $\varphi : \mathfrak{S}_d \to M$ satisfying the following condition: 
$$\varphi(\sigma \tau) = \varphi(\sigma) + ^{\sigma}\!\!\varphi(\tau) + \pi \circ \varpi_4(\tilde{\sigma}\tilde{\tau}(\widetilde{\sigma\tau})^{-1}).$$

\end{dfn}

It is easy to check from Definition \ref{dfn quasi-cocycle} that any quasi-cocycle takes the trivial permutation to $0 \in M$.  Our motivation for introducing quasi-cocycles is the following lemma, which says that finding a subgroup of $\varpi_{4 \to 2}^{-1}(\mathfrak{S}_d)$ surjecting onto $\mathfrak{S}_d$ whose intersection with $\varpi_4(\Gamma(2))$ coincides with $N$ amounts to furnishing a quasi-cocycle.

\begin{lemma} \label{lemma quasi-cocycle}

Suppose that $H \subset \varpi_{4 \to 2}^{-1}(\mathfrak{S}_d)$ is a subgroup generated by a set $\{u_{\sigma}\}_{\sigma \in \mathfrak{S}_d}$ with $\varpi_{4 \to 2}(u_{\sigma}) = \sigma$.  Let $\varphi : \mathfrak{S}_d \to M$ be the map given by $\sigma \mapsto \pi(u_{\sigma} \varpi_4(\tilde{\sigma})^{-1})$.  Then we have $H \cap \varpi_4(\Gamma(2)) = N$ if and only if $\varphi$ is a quasi-cocycle.  In fact, there is a one-to-one correspondence between quasi-cocycles and such subgroups $H \subset \varpi_{4 \to 2}^{-1}(\mathfrak{S}_d)$: this correspondence is given by sending a quasi-cocycle $\varphi : \mathfrak{S}_d \to M$ to the subgroup generated by the elements $\phi_{\sigma} \varpi_4(\tilde{\sigma})$, where $\phi_{\sigma} \in \varpi_4(\Gamma(2))$ is any element such that $\pi(\phi_{\sigma}) = \varphi(\sigma)$ (such a subgroup $H$ is well defined regardless of our choices of $\phi_{\sigma} \in \pi^{-1}(\varphi(\sigma))$ since Lemma \ref{lemma contains N} implies that $H \supset N = \ker(\pi)$).

\end{lemma}

\begin{proof}

Let $H \subset \varpi_{4 \to 2}^{-1}(\mathfrak{S}_d)$ be a subgroup generated by elements $u_{\sigma}$ as in the statement.  Suppose that we have $H \cap \varpi_4(\Gamma(2)) = N$.  For each $\sigma \in \mathfrak{S}_d$, let $\phi_{\sigma} = u_{\sigma} \varpi_4(\tilde{\sigma})^{-1}$, and let $\varphi : \mathfrak{S}_d \to M$ be the map given by $\sigma \mapsto \pi(\phi_{\sigma})$.  Let $\sigma, \tau$ be any two elements of $\mathfrak{S}_d$.  Then we have 
\begin{equation}
u_{\sigma} u_{\tau}u_{\sigma\tau}^{-1} =  \phi_{\sigma} \varpi_4(\tilde{\sigma})\phi_{\tau} \varpi_4(\tilde{\tau})\varpi_4(\widetilde{\sigma\tau})^{-1}\phi_{\sigma\tau}^{-1}.
\end{equation}
  Since $\varpi_{4 \to 2}(u_{\sigma} u_{\tau}u_{\sigma\tau}^{-1}) = 1 \in \mathfrak{S}_d$, we have $u_{\sigma} u_{\tau}u_{\sigma\tau}^{-1} \in H \cap \varpi_4(\Gamma(2)) = N$ and thus, we have 
\begin{equation}
\phi_{\sigma\tau} \equiv \phi_{\sigma} \varpi_4(\tilde{\sigma})\phi_{\tau} \varpi_4(\tilde{\tau})\varpi_4(\widetilde{\sigma\tau})^{-1} = \phi_{\sigma}\ \!\! ^{\sigma}\!\phi_{\tau} \varpi_4(\tilde{\sigma})\varpi_4(\tilde{\tau})\varpi_4(\widetilde{\sigma\tau})^{-1} \ \ \ \ (\mathrm{mod} \ N).
\end{equation}
  Writing this as an equation of the images in $M$ yields precisely the condition in Definition \ref{dfn quasi-cocycle}.

Conversely, suppose that $\varphi : \mathfrak{S}_d \to M$ is a quasi-cocycle.  Let $u \in H$ be any element with $\varpi_{4 \to 2}(u) = 1$.  Then $u$ can be written as a product $u_{\sigma_1} ... u_{\sigma_r}$ for some (not necessarily distinct) elements $\sigma_1, ... , \sigma_r \in \mathfrak{S}_d$ with $\sigma_1 ... \sigma_r = 1$.  It follows from writing $u_{\sigma_i}\varpi_4(\tilde{\sigma})^{-1} \in \varphi(\sigma_i) + N$ and by relating $u_{\sigma_1 ... \sigma_r}$ with $u_{\sigma_1} ... u_{\sigma_r}$ by repeated applications of the condition in Definition \ref{dfn quasi-cocycle} that we have  
\begin{equation}
N \ni u_{(1)} = u_{\sigma_1 ... \sigma_r} \equiv u_{\sigma_1} ... u_{\sigma_r} = u \ \ \ \ (\mathrm{mod} \ N),
\end{equation}
 which implies that $u \in N$.  Since $u$ was chosen as an arbitrary element of $H \cap \varpi_4(\Gamma(2))$ and we have $N \subseteq H \cap \varpi_4(\Gamma(2))$ from Lemma \ref{lemma contains N}, we get the desired statement.

Finally, to prove the one-to-one correspondence, we note that by Corollary \ref{cor contains N}, a subgroup $H \subset \varpi_{4 \to 2}^{-1}(\mathfrak{S}_d)$ satisfying $\varpi_{4 \to 2}(H) = \mathfrak{S}_d$ and $H \cap \varpi_4(\Gamma(2)) = N$ is generated by any choice of elements $u_{\sigma} \in \varpi_{4 \to 2}^{-1}(\sigma)$ and therefore is induced by the quasi-cocycle $\varphi : \sigma \mapsto \pi(u_{\sigma} \varpi_4(\tilde{\sigma})^{-1})$.  Moreover, this quasi-cocycle is unique: suppose that $\varphi$ and $\varphi'$ are two quasi-cocycles which give rise to the subgroup $H \subset \varpi_{4 \to 2}^{-1}(\mathfrak{S}_d)$ in this way.  Then for any $\sigma \in \mathfrak{S}_d$, we have that $u_{\sigma} := \phi_{\sigma} \varpi_4(\tilde{\sigma})$ and $u'_{\sigma} := \phi'_{\sigma} \varpi_4(\tilde{\sigma})$ are both elements of $H$ (where $\pi(\phi_{\sigma}) = \varphi(\sigma)$ and $\pi(\phi'_{\sigma}) = \varphi'(\sigma)$), and so we have $\phi'_{\sigma}\phi_{\sigma}^{-1} \in H \cap \varpi_4(\Gamma(2)) = N$.  Thus, we have $\phi_{\sigma} \equiv \phi'_{\sigma}$ (mod $N$), or equivalently, $\varphi(\sigma) = \varphi'(\sigma)$, and so the quasi-cocycles $\varphi$ and $\varphi'$ are equal.

\end{proof}

\begin{lemma} \label{lemma quasi-cocycle2}

Let $\varphi : \mathfrak{S}_d \to M$ be a quasi-cocycle.  Then the map $\varphi$ satisfies the following conditions.

(i) We have $\langle \varphi((i, j)), v_{\{i, j\}} \rangle = 1$ for $1 \leq i < j \leq d$.

(ii) We have $\langle \varphi((i, j)), v_{\{k, l\}} \rangle = 0$ for distinct $i, j, k, l$.

Conversely, suppose that we define a map $\varphi : \mathfrak{S}_d \to M$ as follows.  We first assign values of $\varphi((1, j))$ for $2 \leq j \leq d$ which satisfy conditions (i) and (ii) and then, for each $\sigma \in \mathfrak{S}_d \smallsetminus \{(1, j)\}_{2 \leq j \leq d}$, we write $\sigma$ as a product of the generators $(1, j)$ and apply the condition given in Definition \ref{dfn quasi-cocycle} to determine $\varphi(\sigma)$.  Then the map $\varphi$ is well defined (i.e. it does not depend on the choice of presentation of each $\sigma$ as a product of generators), and $\varphi$ is a quasi-cocycle.

\end{lemma}

\begin{proof}

Given any transposition $\sigma = (i, j)$, we compute 
\begin{equation} \label{eq quasi-cocycle1}
0 = \varphi(1) = \varphi(\sigma^2) = \ ^{\sigma}\!\varphi(\sigma) + \varphi(\sigma)+ \pi \circ \varpi_4(\widetilde{(i, j)}^2) = \langle \varphi(\sigma), v_{\{i, j\}} \rangle v_{\{i, j\}} + v_{\{i, j\}},
\end{equation}
 and therefore (i) holds.  Now given any transpositions $\sigma = (i, j)$ and $\tau = (k, l)$ with $\{i, j\} \cap \{k, l\} = \varnothing$ (so $\sigma$ and $\tau$ commute), the second relation given by Lemma \ref{lemma lift identities} says that the elements $\pi \circ \varpi_4(\tilde{\sigma})$ and $\pi \circ \varpi_4(\tilde{\tau})$ commute.  Let $w = \pi \circ \varpi_4(\tilde{\sigma}\tilde{\tau}(\widetilde{\sigma\tau})^{-1}) = \pi \circ \varpi_4(\tilde{\tau}\tilde{\sigma}(\widetilde{\tau\sigma})^{-1}) \in M$.  We compute that 
\begin{align} \label{eq quasi-cocycle2}
\begin{split}
&\varphi(\sigma\tau) = \varphi(\sigma) + ^{\sigma}\!\!\varphi(\tau) + w = \varphi(\sigma) + \varphi(\tau) + \langle \varphi(\tau), v_{\{i, j\}} \rangle v_{\{i, j\}} + w \\
 = \ &\varphi(\tau\sigma) = \varphi(\tau) + ^{\tau}\!\!\varphi(\sigma) + w = \varphi(\tau) + \varphi(\sigma) + \langle \varphi(\sigma), v_{\{k, l\}} \rangle v_{\{k, l\}} + w.
\end{split}
\end{align}
  Thus, we have $\langle \varphi(\tau), v_{\{i, j\}} \rangle v_{\{i, j\}} - \langle \varphi(\sigma), v_{\{k, l\}} \rangle v_{\{k, l\}} = \varphi(\sigma\tau) - \varphi(\tau\sigma) = 0$.  Since $v_{\{i, j\}}$ and $v_{\{k, l\}}$ are linearly independent, this implies that $\langle \varphi(\tau), v_{\{i, j\}} \rangle = \langle \varphi(\sigma), v_{\{k, l\}} \rangle = 0$, and therefore (ii) holds.

Now suppose that we have constructed a map $\varphi : \mathfrak{S}_d \to M$ according to the procedure described in the converse statement.  We recall that a presentation for the symmetric group $S_d$ is given by the generators $(1, 2), ... , (1, d)$ and relations $(1, j)^2 = 1$ for $2 \leq j \leq d$ and $(1, j)(1, k)(1, j) = (1, k)(1, j)(1, k)$ for $2 \leq j < k \leq d$.  Therefore, showing that the map $\varphi$ does not depend on particular presentations of each element $\sigma \in \mathfrak{S}_d \smallsetminus \{(1, j)\}_{2 \leq j \leq d}$ as a product of the $(1, j)$'s amounts to showing that (a) we get the same value for $\varphi((1, j)^2)$ for $2 \leq j \leq d$ (which must be $\varphi((1)) = 0$), and (b) we get the same value for $\varphi((1, j)(1, k)(1, j))$ and $\varphi((1, k)(1, j)(1, k))$ for $1 \leq j < k \leq d$.  We have proven above that (a) is equivalent to condition (i), so we set out to prove (b).

Fix distinct $j, k \in \{2, ... , d\}$.  By applying the condition in Definition \ref{dfn quasi-cocycle} to the product $(1, j)(1, k)(1, j)$ and expanding, we get 
$$\varphi((1, j)(1, k)(1, j)) - \pi \circ \varpi_4(\widetilde{(1, j)}\widetilde{(1, k)}\widetilde{(1, j)}\widetilde{(j, k)}^{-1}) = \varphi((i, j)) + ^{(1, j)}\!\varphi((1, k)) + ^{(1, j)(1, k)}\!\varphi((1, j))$$
\begin{equation} \label{eq quasi-cocycle3}
 = \varphi((1, k)) + \langle \varphi((1, k)), v_{\{1, j\}} \rangle v_{\{1, j\}} + \langle \varphi((1, j)), v_{\{1, k\}} \rangle v_{\{1, k\}} + \langle \varphi((1, j)), v_{\{1, j\}} + v_{\{1, k\}} \rangle v_{\{1, j\}}.
\end{equation}
  By doing the same computation with $j$ and $k$ reversed, subtracting the resulting expression from the one in (\ref{eq quasi-cocycle3}), and using the first identity given by Lemma \ref{lemma lift identities}, we compute the difference $\varphi((1, j)(1, k)(1, j)) - \varphi((1, k)(1, j)(1, k))$ to be 
\begin{equation} \label{eq quasi-cocycle4}
w := \varphi((1, k)) + \varphi((1, j)) + \langle \varphi((1, j)), v_{\{j, k\}} \rangle v_{\{1, j\}} + \langle \varphi((1, k)), v_{\{j, k\}} \rangle v_{\{1, k\}}.
\end{equation}
  We claim that $w = 0 \in M$, as desired.  Due to the nondegeneracy of the pairing $\langle \cdot, \cdot \rangle$ on $M$, this follows from a verification that that $w$ lies in the kernel of the pairing.  It clearly suffices to check that $\langle w, v_I \rangle = 0$ for any $2$-element subset $I \subseteq \{1, ... , d\}$.  Using bilinearity and condition (i), we compute $\langle w, v_{\{1, j\}} \rangle$ to be 
\begin{equation}
\langle \varphi((1, k)), v_{\{1, j\}} \rangle + \langle \varphi((1, j)), v_{\{1, j\}} \rangle + \langle \varphi((1, k)), v_{\{j, k\}} \rangle = \langle \varphi((1, k)), v_{\{1, j\}} + v_{\{j, k\}} \rangle + 1 = 0.
\end{equation}
  By symmetry, we also get $\langle w, v_{\{1, k\}} \rangle = 0$.  Now for any index $i \notin \{1, j, k\}$, using bilinearity and condition (ii), we compute $\langle w, v_{\{i, j\}} \rangle$ to be 
\begin{equation}
\langle \varphi((1, k)), v_{\{i, j\}} \rangle + \langle \varphi((1, j)), v_{\{i, j\}} \rangle + \langle \varphi((1, j)), v_{\{j, k\}} \rangle = 0 + \langle \varphi((1, j)), v_{\{i, k\}} \rangle = 0.
\end{equation}
  By symmetry, we also get $\langle w, v_{\{i, k\}} \rangle = 0$ as well.  Finally, for any distinct indices $i, l \notin \{1, j, k\}$, it is clear that $\langle w, v_{\{i, l\}} \rangle = 0$ from the fact that $v_{\{i, l\}}$ is orthogonal to each term in the expression for $w$ in (\ref{eq quasi-cocycle4}) (here we again use condition (ii)), and our claim is proved.

The map $\varphi$ which we have obtained in the above way is clearly a quasi-cocycle by construction, and the proposition is proved.

\end{proof}

\subsection{The possible images of $G$ modulo $4$ when $d$ is even} \label{sec3.3}

In this subsection we will show that, under the hypotheses of Theorem \ref{thm main} and the assumption that $d = 2g + 2$, the subgroup $G \subset \Sp(V)$ must be the full preimage $\varpi_4^{-1}(\mathfrak{S}_d)$.  We first need the following simple lemma.

\begin{lemma} \label{lemma nonexistence}

There does not exist a quasi-cocycle $\mathfrak{S}_{2g + 2} \to M$.

\end{lemma}

\begin{proof}

Suppose that we have a quasi-cocycle $\varphi : \mathfrak{S}_{2g + 2} \to M$.  Then condition (ii) in the statement of Lemma \ref{lemma quasi-cocycle2} implies that $\langle \varphi((1, 2)), \sum_{3 \leq i < j \leq 2g + 2} v_{\{i, j\}} \rangle = 0$.  But since we have $v_{\{1, 2\}} = \sum_{3 \leq i < j \leq 2g + 2} v_{\{i, j\}}$ by Lemma \ref{lemma compliment}, this contradicts condition (i) given by Lemma \ref{lemma quasi-cocycle2}.  Therefore there is no such quasi-cocycle.

\end{proof}

It is now easy to prove Theorem \ref{thm main}(a).

\begin{proof}[Proof of Theorem \ref{thm main}(a)]

Assume that $d = 2g + 2$.  Lemma \ref{lemma sufficient} says that it suffices to show that the group $H := \varpi_4(G)$ satisfies $H \supset \varpi_4(\Gamma(2))$.  Suppose that $H$ does not contain $\varpi_4(\Gamma(2))$.  Then it follows from Lemma \ref{lemma contains N} that we must have $H \cap \varpi_4(\Gamma(2)) = N$, and so $H$ corresponds to a quasi-cocycle $\varphi : \mathfrak{S}_{2g + 2} \to M$ as given by Lemma \ref{lemma quasi-cocycle}.  Since such a quasi-cocycle does not exist by Lemma \ref{lemma nonexistence}, we get a contradiction.

\end{proof}

\subsection{The possible images of $G$ modulo $4$ when $d$ is odd} \label{sec3.4}

The following theorem classifies the possible modulo-$4$ images of $G$ under the hypotheses of Theorem \ref{thm main} and the assumption that $d = 2g + 1$ and brings us a step closer to proving part (b) of that theorem.

\begin{thm} \label{thm mod 4}

For each $\mathbf{c} = (c_2, ... , c_{2g + 1}) \in \ff_2^{2g}$, let $\varphi_{\mathbf{c}} : \mathfrak{S}_{2g + 1} \to M$ be the map defined by assigning $\varphi_{\mathbf{c}}((1, j)) = c_j v_{\{1, j\}} + v_{\{2, ... , 2g + 1\}}$ for $2 \leq j \leq 2g + 1$ and determining $\varphi_{\mathbf{c}}(\sigma)$ for $\sigma \in \mathfrak{S}_{2g + 1} \smallsetminus \{(1, j)\}_{2 \leq j \leq 2g + 1}$ as in the second statement of Lemma \ref{lemma quasi-cocycle2} (by that lemma, this map is a well-defined quasi-cocycle).  Then every quasi-cocycle $\varphi : \mathfrak{S}_{2g + 1} \to M$ is equal to $\varphi_{\mathbf{c}}$ for some $\mathbf{c} \in \ff_2^{2g}$.  Therefore, there are precisely $2^{2g}$ subgroups $H \subset \varpi_{4 \to 2}^{-1}(\mathfrak{S}_{2g + 1})$ satisfying $\varpi_{4 \to 2}(H) = \mathfrak{S}_{2g + 1}$ and $H \cap \varpi_4(\Gamma(2)) = N$.  Moreover, the set of these subgroups forms a full conjugacy class of subgroups of $\varpi_{4 \to 2}^{-1}(\mathfrak{S}_{2g + 1})$.

\end{thm}

\begin{proof}

It is immediate to verify that the $2^{2g}$ (distinct) maps $\varphi_{\mathbf{c}} : \mathfrak{S}_{2g + 1} \to M$ satisfy conditions (i) and (ii) given by Lemma \ref{lemma quasi-cocycle2} for the transpositions $(1, j)$ and are therefore quasi-cocycles.  Now let $\varphi$ be any quasi-cocycle satisfying those conditions (i) and (ii) for the transpositions $(1, j)$.  Then for each $j$, those conditions (i), (ii) imply that $\langle (\varphi - \varphi_{\mathbf{0}})((1, j)), v_{\{1, j\}} \rangle = \langle (\varphi - \varphi_{\mathbf{0}})((1, j)), v_{\{k, l\}} \rangle = 0$ for all $k, l \notin \{1, j\}$.  Thus, the vector $(\varphi - \varphi_{\mathbf{0}})((1, j)) \in M$ lies in the orthogonal complement of the subspace of $M$ generated by $\{v_{\{1, j\}}\} \cup \{v_{\{k, l\}}\}_{k, l \notin \{1, j\}}$, which clearly coincides with the subspace spanned by the vector $v_{\{1, j\}}$.  It follows that $(\varphi - \varphi_{\mathbf{0}})((1, j)) = c_j v_{\{1, j\}}$ for some $c_j \in \ff_2$, which implies that $\varphi = \varphi_{\mathbf{c}}$ where $\mathbf{c} = (c_2, ... , c_{2g + 1})$.  The fact that these correspond to $2^{2g}$ distinct subgroups $H$ as in the statement now follows from Lemma \ref{lemma quasi-cocycle}.

We now prove the statement about conjugacy.  It is clear that conjugating any subgroup $H$ satisfying $\varpi_{4 \to 2}(H) = \mathfrak{S}_{2g + 1}$ and $H \cap \varpi_4(\Gamma(2)) = N$ by any element in $\varpi_{4 \to 2}^{-1}(\mathfrak{S}_{2g + 1})$ yields another subgroup with those same properties (here we are using the fact that the $S_{2g + 1}$-invariance of $N$ implies that $N \lhd \varpi_{4 \to 2}^{-1}(\mathfrak{S}_{2g + 1})$).  It remains to show that all such subgroups are conjugate.  For each $\mathbf{c} \in \ff_2^{2g}$, we denote the subgroup of $\varpi_2^{-1}(\mathfrak{S}_d)$ corresponding to $\varphi_{\mathbf{c}}$ via Lemma \ref{lemma quasi-cocycle} by $H_{\mathbf{c}}$; we shall prove that the subgroups $H_{\mathbf{c}}$ are all conjugate.  Fix a vector $\mathbf{c} \in \ff_2^{2g}$ and choose a set of generators $u_{\sigma}$ for $H_{\mathbf{c}}$ with $\pi(u_{\sigma}\varpi_4(\tilde{\sigma})) = \varphi_{\mathbf{c}}(\sigma)$.  For a choice of distinct $i, j \in \{1, ... , 2g + 1\}$, we have $\varpi_4(\widetilde{(i, j)})^2 H_{\mathbf{c}} \varpi_4(\widetilde{(i, j)})^{-2} = H_{\mathbf{c}'}$ for some $\mathbf{c}' = (c'_2, ... , c'_{2g + 1}) \in \ff_2^{2g}$.  We claim that we have $\mathbf{c}' - \mathbf{c} = (\langle v_{\{i, j\}}, v_{\{1, k\}} \rangle)_{2 \leq k \leq 2g + 1}$.  Indeed, for $k \in \{2, ... , 2g + 1\}$, setting $\phi_{(1, k)} = u_{(1, k)}\varpi_4(\widetilde{(1, k)})$, we compute that $\varpi_4(\widetilde{(i, j)}^2) u_{(1, k)} \varpi_4(\widetilde{(i, j)}^{-2}) = \varpi_4(\widetilde{(i, j)}^2) \phi_{(1, k)} \varpi_4(\widetilde{(1, k)} \widetilde{(i, j)}^{-2})$ 
\begin{equation} \label{eq conjugacy class}
 = \varpi_4(\widetilde{(i, j)}^2) \phi_{(1, k)} \varpi_4 \big( \widetilde{(1, k)}\widetilde{(i, j)}^{-2}\widetilde{(1, k)}^{-1}) \varpi_4(\widetilde{(1, k)}) = \varpi_4 \big(\widetilde{(i, j)}^2(\widetilde{(1, k)}\widetilde{(i, j)}^{-2}\widetilde{(1, k)}^{-1}) \big) u_{(1, k)}.
\end{equation}
  We know that $\varpi_4(\widetilde{(i, j)})^2 u_{(1, k)} \varpi_4(\widetilde{(i, j)})^{-2}$ has image $(1, k) \in \mathfrak{S}_d$ under $\pi$ and therefore is equal to $\phi'_{(1, k)} \varpi_4(\widetilde{(1, k)}) \in H_{\mathbf{c}'}$ for some $\phi'_{(1, k)} \in \varpi_4(\Gamma(2))$ with $\pi(\phi'_{(1, k)}) = \varphi_{\mathbf{c}'}((1, k)) \in M$.  Multiplying the expression on the right-hand side of (\ref{eq conjugacy class}) by $\varpi_4(\widetilde{(1, k)})^{-1}$ and reducing modulo $N$ to get an expression for $\varphi_{\mathbf{c}'}((1, k))$ yields $v_{\{i, j\}} + ^{(1, k)}\!\!v_{\{i, j\}} + \varphi((1, k)) = \langle v_{\{i, j\}}, v_{\{1, k\}} \rangle v_{\{1, k\}} + \varphi((1, k)) \in M$.  We therefore have $c'_k - c_k = \langle v_{\{i, j\}}, v_{\{1, k\}} \rangle$, thus proving the claim.

The above claim implies by linearity that for any $u \in \varpi_4(\Gamma(2))$, if $\mathbf{c}' \in \ff_2^{2g}$ is the vector such that $u H_{\mathbf{c}} u^{-1} = H_{\mathbf{c}'}$, then we have $\mathbf{c}' - \mathbf{c} = (\langle \pi(u), v_{\{1, k\}} \rangle)_{2 \leq k \leq 2g + 1}$.  Now it follows from the nondegeneracy of the pairing $\langle \cdot, \cdot \rangle$ on $M$ and the fact that the $v_{\{1, k\}}$'s form a basis for $M$ that given any $\mathbf{c}, \mathbf{c}' \in \ff_2^{2g}$, we may choose an element $u \in \varpi_4(\Gamma(2))$ such that $(\langle \pi(u), v_{\{1, k\}} \rangle)_{2 \leq k \leq 2g + 1} = \mathbf{c}' - \mathbf{c}$ and thus $u H_{\mathbf{c}} u^{-1} = H_{\mathbf{c}'}$.  The desired statement is proved.

\end{proof}

We conclude the section with the following remark, which will be useful in \S\ref{sec4.2}.

\begin{rmk} \label{rmk quasi-cocycle}

For any $\mathbf{c} \in \ff_2^{2g}$, it is easy to show directly from properties (i) and (ii) given in the statement of Lemma \ref{lemma quasi-cocycle2} that the subset $I \subset \{1, ... , 2g + 1\}$ such that $\varpi_{\mathbf{c}}((i, j)) = v_I$ satisfies the property that $\{1, ... , 2g + 1\} \smallsetminus I$ is a singleton subset of $\{i, j\}$.

\end{rmk}

\section{Lifting to $\zz / 8\zz$} \label{sec4}

For the rest of the paper, we assume that $g \geq 2$ and that $d = 2g + 1$.  We retain all previous notation and in particular our fixed choices of $a_{i, j} \in V$ from earlier, with $\widetilde{(i, j)} = t_{a_{i, j}} \in \Sp(V)$.  In this subsection we shall determine all lifts to $\Sp(V / 8V)$ of each of the subgroups we classified by Theorem \ref{thm mod 4}(b).  We begin by determining the possible intersections of such lifts with $\varpi_8(\Gamma(2))$.

\subsection{The possible images of $G \cap \Gamma(2)$ modulo $8$} \label{sec4.1}

We first show that the only possibilities for the subgroup $\varpi_8(G \cap \Gamma(4)) \subseteq \varpi_8(\Gamma(4))$ are that it is the full $\varpi_8(\Gamma(4))$ or that it coincides with $N^{(4)}$.

\begin{lemma} \label{lemma contains N mod 8}

Suppose that $H \subset \varpi_{8 \to 2}^{-1}(\mathfrak{S}_d)$ is a subgroup satisfying $\varpi_{8 \to 2}(\tilde{H}) = \mathfrak{S}_d$.  We have the (strict) containment $N^{(4)} \subset H$.  In fact, we have that $H \not\supset \varpi_8(\Gamma(4))$ implies $H \cap \varpi_8(\Gamma(4)) = N^{(4)}$.

\end{lemma}

\begin{proof}

Exactly as in the proof of the analogous statement of Lemma \ref{lemma contains N}, we have that $H \cap \varpi_8(\Gamma(4))$ is an $S_d$-invariant subspace of $\varpi_8(\Gamma(4))$ (a fact which we will use freely below) and so the second statement follows from the first.

We know from Lemma \ref{lemma contains N} that $\varpi_{8 \to 4}(H) \cap \varpi_4(\Gamma(2)) \supseteq N^{(2)}$.  For ease of notation, for $1 \leq i < j \leq 2g + 1$ we write $\lfloor i, j \rfloor = \lfloor j, i \rfloor \in \varpi_8(\Gamma(2))$ for the image modulo $8$ of $t_{a_{i, j}}^2$.  Choose elements $\widetilde{\Delta}_{1, 2, 3}, \widetilde{\Delta}_{1, 2, 4} \in H \cap \varpi_8(\Gamma(2))$ whose respective images modulo $4$ are $\Delta_{1, 2, 3}, \Delta_{1, 2, 4} \in N^{(2)}$.  Then we can write 
\begin{equation}
\widetilde{\Delta}_{1, 2, 3} = s\lfloor 1, 2 \rfloor \lfloor 1, 3 \rfloor \lfloor 2, 3 \rfloor, \ \widetilde{\Delta}_{1, 2, 4} = s'\lfloor 1, 2 \rfloor \lfloor1, 4 \rfloor \lfloor 2, 4 \rfloor
\end{equation}
 for some $s, s' \in H \cap \varpi_8(\Gamma(4))$.  Now we use Proposition \ref{prop N}(e) to compute that the commutator $\widetilde{\Delta}_{1, 2, 3} \widetilde{\Delta}_{1, 2, 4} \widetilde{\Delta}_{1, 2, 3}^{-1} \widetilde{\Delta}_{1, 2, 4}^{-1} \in H$ is 
\begin{equation}
\Delta_{1, 2, 3} + \Delta_{1, 2, 3} + \Delta_{1, 2, 4} + \Delta_{1, 3, 4} + \Delta_{1, 2, 4} + \Delta_{2, 3, 4} = \Delta_{1, 3, 4} + \Delta_{2, 3, 4} \in \varpi_8(\Gamma(4)).
\end{equation}
  We therefore get that $H$ contains all elements of the form $\Delta_{i, j, k} + \Delta_{l, j, k} \in \varpi_8(\Gamma(4))$ for distinct $i, j, k, l$.  It is easy to see that the set of all such elements generates $N^{(4)}_0$, and so we have $H \supseteq N^{(4)}_0$.  By Proposition \ref{prop N}(c), it now suffices to show that $H$ contains some element of $\varpi_8(\Gamma(4)) \smallsetminus \varpi_8(\Gamma(4))_0$.

If we have $\varpi_{8 \to 4}(H) \supset \varpi_4(\Gamma(2))$, then by Lemma \ref{lemma sufficient}, we have $H \supset \varpi_8(\Gamma(4))$ and we are done.  We therefore assume otherwise, which implies by Lemma \ref{lemma contains N} that we have the equality $\varpi_{8 \to 4}(H \cap \varpi_8(\Gamma(2))) = N^{(2)}$.  Lemma \ref{lemma quasi-cocycle} and Theorem \ref{thm mod 4}(b) imply that $\varpi_{8 \to 4}(H)$ is conjugate in $\varpi_{4 \to 2}^{-1}(\mathfrak{S}_{2g + 1})$ to a subgroup corresponding to the quasi-cocycle $\varphi_{\mathbf{0}}$ where $\mathbf{0} \in \ff_2^{2g}$ is the zero vector.  We therefore assume, after possibly conjugating by a suitable element of $\varpi_{8 \to 2}^{-1}(\mathfrak{S}_{2g + 1})$, that $\varpi_{8 \to 4}(H) \subset \varpi_{4 \to 2}^{-1}(\mathfrak{S}_{2g + 1})$ is the subgroup corresponding to the quasi-cocycle $\varphi_{\mathbf{0}}$ via the constructions given in \S\ref{sec3}.  Since we have $\pi \circ \varpi_{8 \to 4}(\lfloor 1, 2 \rfloor ... \lfloor 1, 2g + 1 \rfloor) = v_{\{2, ... , 2g + 1\}} \in M^{(2)}$, the elements $\varpi_{8 \to 4}(\lfloor 1, 2 \rfloor ... \lfloor 1, 2g + 1 \rfloor) \varpi_4(\widetilde{(1, j)})$ lie in $\varpi_{8 \to 4}(H)$ for $2 \leq j \leq 2g + 1$.  Choose respective liftings $u_{(1, j)} \in H$ of these elements, so that we have $u_{(1, j)} = s_j \lfloor 1, 2 \rfloor ... \lfloor 1, 2g + 1 \rfloor \varpi_8(\widetilde{(1, j)})$ for each $j$, where the $s_j$'s are elements of $\varpi_8(\Gamma(4))$.  We claim that there exist $j, k \in \{2, ... , 2g + 1\}$ such that we have $u_{(1, j)}u_{(1, k)}u_{(1, j)}u_{(1, k)}^{-1}u_{(1, j)}^{-1}u_{(1, k)}^{-1} \in \varpi_8(\Gamma(4)) \smallsetminus \varpi_8(\Gamma(4))_0$, thus proving the lemma.

For $2 \leq j \leq 2g + 1$, let $u_{(1, j)}' = s_{j}^{-1} u_{(1, j)} = \lfloor 1, 2 \rfloor ... \lfloor 1, 2g + 1 \rfloor \varpi_8(\widetilde{(1, j)})$.  Using Proposition \ref{prop N}(e), we see that $u_{(1, j)}'u_{(1, k)}'u_{(1, j)}'u_{(1, k)}'^{-1}u_{(1, j)}'^{-1}u_{(1, k)}'^{-1} = s_j^{-3} s_k^{-3} u_{(1, j)}u_{(1, k)}u_{(1, j)}u_{(1, k)}^{-1}u_{(1, j)}^{-1}u_{(1, k)}^{-1}$.  Since $g \geq 2$, there is some choice of $j, k$ such that $s_j \equiv s_k$ (mod $\varpi_8(\Gamma(4))_0$), and in this case we have $u_{(1, j)}'u_{(1, k)}'u_{(1, j)}'u_{(1, k)}'^{-1}u_{(1, j)}'^{-1}u_{(1, k)}'^{-1} \equiv u_{(1, j)}u_{(1, k)}u_{(1, j)}u_{(1, k)}^{-1}u_{(1, j)}^{-1}u_{(1, k)}^{-1}$ (mod $\varpi_8(\Gamma(4))_0$).  It therefore suffices to show that $u_{(1, j)}'u_{(1, k)}'u_{(1, j)}'u_{(1, k)}'^{-1}u_{(1, j)}'^{-1}u_{(1, k)}'^{-1} \in \varpi_8(\Gamma(4)) \smallsetminus \varpi_8(\Gamma(4))_0$ for $2 \leq j < k \leq 2g + 1$.

By slightly abusing notation for the sake of brevity, below we use the superscript $^{\widetilde{(i, j)}}$ to indicate conjugation by $\varpi_8(\widetilde{(i, j)})$ for distinct $i, j$.  We also write $\mu$ for $\lfloor 1, 2 \rfloor ... \lfloor 1, 2g + 1 \rfloor \in \varpi_8(\Gamma(2))$.

Since $\varpi_8(\Gamma(4))_0$ is normal in $\varpi_{8 \to 2}^{-1}(\mathfrak{S}_{2g + 1})$, it suffices to show for any $j, k$ that we have 

\noindent $\mu^{-1}u_{(1, j)}'u_{(1, k)}'u_{(1, j)}'u_{(1, k)}'^{-1}u_{(1, j)}'^{-1}u_{(1, k)}'^{-1}\mu \in \varpi_8(\Gamma(4)) \smallsetminus \varpi_8(\Gamma(4))_0$.  Using the first relation given by Lemma \ref{lemma lift identities}, we compute 
\begin{equation} \label{eq long1}
\mu^{-1}u_{(1, j)}'u_{(1, k)}'u_{(1, j)}'u_{(1, k)}'^{-1}u_{(1, j)}'^{-1}u_{(1, k)}'^{-1}\mu = (^{\widetilde{(1, j)}}\!\mu) (^{\widetilde{(1, j)}\widetilde{(1, k)}}\!\mu) (^{\widetilde{(1, k)}\widetilde{(1, j)}}\!\mu)^{-1} (^{\widetilde{(1, k)}}\!\mu)^{-1}.
\end{equation}
  We now proceed to show that this element lies in $N^{(4)} \smallsetminus N^{(4)}_0 \subset \varpi_8(\Gamma(4)) \smallsetminus \varpi_8(\Gamma(4))_0$, freely using the fact that elements of $N^{(4)}$ commute with everything in $\varpi_8(\Gamma(2))$ by Proposition \ref{prop N}(e).

\textbf{Step 1:} We show that the expression in (\ref{eq long1}) is equivalent modulo $N^{(4)}_0$ to 
\begin{equation} \label{eq long2}
\mu_j (^{\widetilde{(1, j)}}\!\mu_k) (^{\widetilde{(1, k)}}\!\mu_j)^{-1} \mu_k^{-1},
\end{equation}
  where $\mu_i$ denotes $\lfloor i, 2 \rfloor ... \lfloor i, i - 1 \rfloor \lfloor i, 1 \rfloor \lfloor i, i + 1 \rfloor ... \lfloor i, 2g + 1 \rfloor$ for $2 \leq i \leq 2g + 1$.  We do this by showing that $\nu_i := (^{\widetilde{(1, i)}}\!\mu) \mu_i^{-1} \in N^{(4)}$ for any $i$; we will then be able to factor $\nu_j (^{\widetilde{(1, j)}}\!\nu_k) (^{\widetilde{(1, k)}}\!\nu_j)^{-1} \nu_k^{-1} \in N^{(4)}_0$ from the expression in (\ref{eq long1}) to get the expression in (\ref{eq long2}) multiplied by an element of $N^{(4)}_0$.  Fix a choice of $i \in \{2, ... , 2g + 1\}$.  We first observe that for $l \neq i$, we have 
\begin{equation}
^{\widetilde{(1, i)}}\!\lfloor 1, l \rfloor \lfloor i, l \rfloor^{-1} = \varpi_8(t_{a_{1, i}} t_{a_{1, l}}^2 t_{a_{1, i}}^{-1} t_{a_{i, l}}^{-2}) = \varpi_8(t_{t_{a_{1, i}}(a_{1, l})}^2 t_{a_{i, l}}^{-2}),
\end{equation}
 and that since $t_{a_{1, i}}(a_{1, l}) \equiv a_{i, l}$ (mod $2$), by Proposition \ref{prop N}(f) we have $^{\widetilde{(1, i)}}\!\lfloor 1, l \rfloor \lfloor i, l \rfloor^{-1} \in N^{(4)}$.  It follows that $^{\widetilde{(1, i)}}\!\mu \equiv \lfloor i, 2 \rfloor ... \lfloor i, i - 1 \rfloor \lfloor i, 1 \rfloor \lfloor i, i + 1 \rfloor ... \lfloor i, 2g + 1 \rfloor = \mu_i$ (mod $N^{(4)}$).

\textbf{Step 2:} We show that $^{\widetilde{(1, k)}}\!\mu_j \equiv \mu_j$ (mod $N^{(4)}_0$) for any distinct $j, k \in \{2, ... , 2g + 1\}$, which  will allow us to reduce the expression in (\ref{eq long2}) to $\mu_j \mu_k \mu_j^{-1} \mu_k^{-1}$.  Since the commutator of any two terms of the form $\lfloor j, l \rfloor$ lies in $N^{(4)}$ by Proposition \ref{prop N}(e), we can reorder the terms in the defining formula for $\mu_j$ and get $\nu_j' := \mu_j' \mu_j^{-1} \in N^{(4)}$, where $\mu_j' = \lfloor j, k \rfloor \lfloor j, 1 \rfloor \prod_{l \neq 1, j} \lfloor j, l \rfloor$ with the product taken from least to greatest.  Then we have $^{\widetilde{(1, k)}}\!\mu_j \mu_j^{-1} = ^{\widetilde{(1, k)}}\!\!\!\mu_j' \mu_j'^{-1} [(^{\widetilde{(1, k)}}\!\nu_j')\nu_j'^{-1}] \equiv ^{\widetilde{(1, k)}}\!\!\!\mu_j' \mu_j'^{-1}$ (mod $N^{(4)}_0$).

It is now straightforward to check for any $a, b \in V$ with $\langle a, b \rangle \equiv 0$ (mod $2$) that we have the relation $t_a^{-2} t_{t_a(b)}^2 \equiv t_{a + b}^4 t_a^4 t_b^4$ (mod $8$), by appropriately composing the formulas $v \mapsto v \pm \langle v, w \rangle w$ which defines the operators $t_w^{\pm 1}$ for $w = a, b, a + b$ and comparing the two sides.  Then for any choice of $l \neq 1, j, k$, if we put $b = a_{j, l}$ and $a = a_{1, k}$, this relation yields 
\begin{equation}
\lfloor j, l \rfloor^{-1} (^{\widetilde{(1, k)}}\!\lfloor j, l \rfloor) = \varpi_8(t_{a_{j, l} + a_{1, k}}^4 t_{a_{j, l}}^4 t_{a_{1, k}}^4),
\end{equation}
 which we identify with $T_{\{1, j, k, l\}} + T_{\{j, l\}} + T_{\{1, k\}} \in \mathfrak{sp}(V / 2V) \cong \varpi_8(\Gamma(4))$.  Using Lemma \ref{lemma compliment}, we see that this equals $[1, j] + [1, l] + [k, j] + [k, l] \in N^{(4)}_0$, so we have $^{\widetilde{(1, k)}}\!\mu_j' \mu_j'^{-1} \equiv \ ^{\widetilde{(1, k)}}\!(\lfloor j, 1 \rfloor \lfloor j, k \rfloor) (\lfloor j, k \rfloor \lfloor 1, j \rfloor)^{-1}$ (mod $N^{(4)}_0$).  Since $t_{a_{1, j}}(a_{1, k}) \equiv a_{j, k}$ (mod $2$), we have $s_{j, k} := t_{a_{1, j}} \lfloor 1, k \rfloor t_{a_{1, j}}^{-1} \lfloor j, k \rfloor^{-1} = t_{t_{a_{1, j}}(a_{1, k})}^2 t_{a_{j, k}}^{-2}$ lies in $N^{(4)}$ by Proposition \ref{prop N}(f).  We are therefore able to compute (using the first identity given by Lemma \ref{lemma lift identities})
$$^{\widetilde{(1, k)}}\!(\lfloor j, k \rfloor \lfloor 1, j \rfloor) = ^{\widetilde{(1, k)}}\!\!(s_{j, k}^{-1} t_{a_{1, j}} \lfloor 1, k \rfloor t_{a_{1, j}}^{-1} \lfloor 1, j \rfloor) = ^{\widetilde{(1, k)}}\!\!s_{j, k}^{-1} t_{a_{1, k}}t_{a_{1, j}}t_{a_{1, k}}^2 t_{a_{1, j}}t_{a_{1, k}}^{-1}$$
$$ = ^{\widetilde{(1, k)}}\!\!s_{j, k}^{-1} t_{a_{1, j}}t_{a_{1, k}}t_{a_{1, j}}t_{a_{1, k}}t_{a_{1, j}}t_{a_{1, k}}^{-1} = ^{\widetilde{(1, k)}}\!\!s_{j, k}^{-1} t_{a_{1, j}}t_{a_{1, k}}^2 t_{a_{1, j}} = ^{\widetilde{(1, k)}}\!\!s_{j, k}^{-1} t_{a_{1, j}} \lfloor 1, k \rfloor t_{a_{1, j}}^{-1} \lfloor 1, j \rfloor$$
\begin{equation}
 = ^{\widetilde{(1, k)}}\!\!s_{j, k}^{-1} s_{j, k} \lfloor j, k \rfloor \lfloor 1, j \rfloor \equiv \lfloor j, k \rfloor \lfloor 1, j \rfloor \ \ \ \ (\mathrm{mod} \ N^{(4)}_0).
\end{equation}
  We therefore get $^{\widetilde{(1, k)}}\!\mu_j \mu_j^{-1} \equiv ^{\widetilde{(1, k)}}\!\!(\lfloor j, 1 \rfloor \lfloor j, k \rfloor) (\lfloor j, k \rfloor \lfloor 1, j \rfloor)^{-1} \equiv 0$ (mod $N^{(4)}_0$), as desired.

\textbf{Step 3:} Finally, we show that the commutator $\mu_j \mu_k \mu_j^{-1} \mu_k^{-1}$ lies in $N^{(4)} \smallsetminus N^{(4)}_0$, which will conclude the proof.  Using the fact that the commutator subgroup of $\varpi_8(\Gamma(2))$ is contained in the center of $\varpi_8(\Gamma(2))$ as implied by Proposition \ref{prop N}(e), we have $\mu_j \mu_k \mu_j^{-1} \mu_k^{-1} = \prod_{l \neq k} (\mu_j \lfloor k, l \rfloor \mu_j^{-1} \lfloor k, l \rfloor^{-1})$.  We further deduce using Proposition \ref{prop N}(e) that for $l \neq j, k$ that $\mu_j \lfloor k, l \rfloor \mu_j^{-1} \lfloor k, l \rfloor^{-1}$ is $\Delta_{j, k, l} + \Delta_{j, k, l} = 0 \in \varpi_8(\Gamma(4))$, while $\mu_j \lfloor j, k \rfloor \mu_j^{-1} \lfloor j, k \rfloor^{-1}$ is computed to be $\sum_{i \neq j, k} \Delta_{i, j, k} \in N^{(4)} \smallsetminus N^{(4)}_0$, and we are done.

\end{proof}

\begin{rmk}

The step in the above proof where we choose $j$ and $k$ so that $s_j \equiv s_k$ (mod $\varpi_8(\Gamma(4))_0$) is the only step of the argument for which the assumption that $g \geq 2$ is necessary.  When $g = 1$, there indeed exist subgroups $H \subset \varpi_{8 \to 2}^{-1}(\mathfrak{S}_3)$ with $\varpi_{8 \to 2}(H) = \mathfrak{S}_3$ and $H \cap \varpi_8(\Gamma(4)) = N^{(4)}_0 = \{1\}$ (e.g. letting $H$ be generated by $\varpi_8(t_{a_{1, 2}}^{-2} t_{a_{1, 3}}^{-2} t_{a_{1, 2}})$ and $\varpi_8(t_{a_{2, 3}}^4 t_{a_{1, 2}}^2 t_{a_{1, 3}}^3)$, we have $H \cap \varpi_8(\Gamma(2)) = \{\pm 1\}$).  If we impose an additional assumption in the $g = 1$ case that $H \cap \varpi_8(\Gamma(4)) = N^{(4)}$, the arguments in the rest of \S\ref{sec4} still carry through and such subgroups can still be classified as in the statement of Theorem \ref{thm mod 8} below.

\end{rmk}

Now that we have determined the subgroup $H \cap \varpi_8(\Gamma(4)) \subset H$ in the case that $H$ does not contain $\varpi_8(\Gamma(4))$, we are able to determine the only possible subgroup $H \cap \varpi_8(\Gamma(2)) \subset H$ in this case.  We retain the notation of $\lfloor i, j \rfloor = \varpi_8(t_{a_{i, j}}^2)$ used in the proof of Lemma \ref{lemma contains N mod 8} and define $\widetilde{N} \subset \varpi_8(\Gamma(2))$ to be the subgroup generated by the elements
$$\delta_{i, j, k} := \lfloor i, j \rfloor \lfloor i, k \rfloor \lfloor j, k \rfloor \prod_{\stackrel{1 \leq l < m \leq 2g + 1}{\{l, m\} \cap \{i, j, k\} \neq \varnothing}} \lfloor l, m \rfloor^2$$
 (note that the elements $\lfloor l, m \rfloor^2$ in the above expression commute with everything in $\varpi_8(\Gamma(2))$ by Proposition \ref{prop N}(e), so in particular there is no need to specify any order).

We write $\widetilde{M}$ for the quotient $\varpi_8(\Gamma(2)) / \widetilde{N}$ and write $\tilde{\pi} : \varpi_8(\Gamma(2)) / \widetilde{N} \twoheadrightarrow \widetilde{M}$ for the corresponding quotient map.  It is clear from Proposition \ref{prop N}(e) that $\widetilde{M}$ is an abelian group and in fact a $\zz / 4\zz$-module generated by the order-$4$ elements $\overline{\lfloor i, j \rfloor} := \tilde{\pi}(\lfloor i, j \rfloor)$ for $1 \leq i < j \leq 2g + 1$ which satisfy certain relations coming from the definition of the elements $\delta_{i, j, k} \in \varpi_8(\Gamma(2))$.  We therefore use additive notation when expressing the elements of $\widetilde{M}$ in terms of the $\overline{\lfloor i, j \rfloor}$'s.

\begin{prop} \label{prop tilde N}

\

a) The subgroup $\widetilde{N} \subset \varpi_8(\Gamma(2))$ is the unique normal subgroup of $\varpi_{8 \to 2}^{-1}(\mathfrak{S}_{2g + 1})$ whose intersection with $\varpi_8(\Gamma(4))$ coincides with $N^{(4)}$ and whose image modulo $4$ coincides with $N^{(2)}$.

b) The conjugation action on the normal subgroup $\varpi_8(\Gamma(2)) \lhd \varpi_{8 \to 2}^{-1}(\mathfrak{S}_{2g + 1})$ induces an action of $S_d = \mathfrak{S}_{2g + 1} = \varpi_{8 \to 2}^{-1}(\mathfrak{S}_{2g + 1}) / \varpi_8(\Gamma(2))$ on $\widetilde{M} = \varpi_8(\Gamma(2)) / \widetilde{N}$ which is given as follows: any permutation $\sigma \in S_{2g + 1}$ sends each generator $\overline{\lfloor i, j \rfloor}$ to $\overline{\lfloor \sigma(i), \sigma(j) \rfloor}$.

\end{prop}

\begin{proof}

We first show that we have $\widetilde{N} \lhd \varpi_{8 \to 2}^{-1}(\mathfrak{S}_{2g + 1})$.  It is clear from Proposition \ref{prop 2-group}(a) and the fact that $S_{2g + 1}$ is generated by transpositions that the group $\varpi_{8 \to 2}^{-1}(\mathfrak{S}_{2g + 1})$ is generated by the lifts $\varpi_8(\widetilde{(i, j)})$.  As in the proof of Lemma \ref{lemma contains N mod 8}, we check using Proposition \ref{prop N}(f) that for $1 \leq i < j \leq 2g + 1$, conjugation by the image modulo $8$ of the lift $\tilde{\sigma} = \widetilde{(i, j)}$ of $\sigma := (i, j)$ sends $\lfloor k, l \rfloor$ to $\lfloor \sigma(k), \sigma(l) \rfloor$ times an element of $N^{(4)} \subset \widetilde{N}$, which proves the normality statement.

We next observe that $\varpi_{8 \to 4}(\delta_{i, j, k}) = \Delta_{i, j, k} \in N^{(2)}$ for all $i, j, k$.  Since these elements generate $N^{(2)}$ by Proposition \ref{prop N}(a), we get that $\varpi_{8 \to 4}(\widetilde{N})$ coincides with $N^{(2)} \subset \varpi_4(\Gamma(2))$.  Moreover, we get the inclusion $N^{(4)} \subseteq \widetilde{N} \cap \varpi_8(\Gamma(4))$ from the easily verifiable fact that ${\delta}_{i, k, j}\delta_{i, j, k}^{-1}$ is equal to $\Delta_{i, j, k} \in \varpi_8(\Gamma(4))$ for any $i, j, k$, again using Proposition \ref{prop N}(a).  To get the reverse inclusion, we refer to the full set of relations among the generators $\Delta_{i, j, k}$ of $N^{(2)}$ given by Proposition \ref{prop N}(a); we need to show that $\delta_{i, j, k}\delta_{i, j, l}\delta_{i, k, l}\delta_{j, k, l} \in N^{(4)}$ for all distinct $i, j, k, l$.  It is straightforward to compute (again using the fact that commutators between elements of the form $\lfloor i, j \rfloor$ lie in $N^{(4)}$ by Proposition \ref{prop N}(e)) that $\delta_{i, j, k}\delta_{i, j, l}\delta_{i, k, l}\delta_{j, k, l}$ is the element 
$$[i, j] + [i, k] + [i, l] + [j, k] + [j, l] + [k, l] + \sum_{m \neq i, j, k, l} ([m, i] + [m, j] + [m, k] + [m, l]) \in \varpi_8(\Gamma(4)),$$
 which belongs to $N^{(4)}$ by definition.  Thus, the subgroup $\widetilde{N} \subset \varpi_8(\Gamma(2))$ satisfies $\widetilde{N} \cap \varpi_8(\Gamma(4)) = N^{(4)}$ and $\varpi_{8 \to 4}(\widetilde{N}) = N^{(2)}$.

To prove uniqueness, suppose that $\widetilde{N}' \subset \varpi_8(\Gamma(2))$ is another subgroup which is normal in $\varpi_{8 \to 2}^{-1}(\mathfrak{S}_{2g + 1})$, whose intersection with $\varpi_8(\Gamma(4))$ coincides with $N^{(4)}$, and whose image modulo $4$ coincides with $N^{(2)}$.  We compare the quotients $\widetilde{N} / N^{(4)}$ and $\widetilde{N}' / N^{(4)}$ as subgroups of $\varpi_8(\Gamma(2)) / N^{(4)}$ as follows.  For all distinct $i, j, k$, we choose lifts $\delta_{i, j, k}' \in \widetilde{N}'$ satisfying $\varpi_4(\delta_{i, j, k}') = \Delta_{i, j, k}$ (noting that this choice is unique modulo $N^{(4)}$), and we let $\epsilon_{i, j, k} \in \varpi_8(\Gamma(2)) / N^{(4)}$ be the image modulo $N^{(4)}$ of $\delta_{i, j, k}' \delta_{i, j, k}^{-1}$.  It follows from Proposition \ref{prop N}(e) that the group $\varpi_8(\Gamma(2)) / N^{(4)}$ is abelian.  From this it is easy to see that the elements $\epsilon_{i, j, k}$ lie in $\varpi_8(\Gamma(4)) / N^{(4)} = M^{(4)}$, and that for distinct $i, j, k, l$, these elements (identified with their images in $M^{(4)}$ and written additively) satisfy $^{\sigma}\!\epsilon_{i, j, k} = \epsilon_{\sigma(i), \sigma(j), \sigma(k)}$ and $\epsilon_{i, j, k} + \epsilon_{i, j, l} + \epsilon_{i, k, l} + \epsilon_{j, k, l} = 0$.  The former property implies that for a given $i,j, k$, the vector $\epsilon_{i, j, k}$ is either $0$ or $v_I \in M$, where $I$ is the even-cardinality set $\{1, ... , 2g + 1\} \smallsetminus \{i, j ,k\}$.  Suppose that $\epsilon_{i, j, k} = v_I$; then for any $l \neq i, j, k$, the former and latter properties above give 
\begin{equation}
0 = \epsilon_{i, j, k} + \epsilon_{i, j, l} + \epsilon_{i, k, l} + \epsilon_{j, k, l} = \epsilon_{i, j, k} + ^{(k, l)}\!\!\epsilon_{i, j, k} + ^{(j, l)}\!\!\epsilon_{i, j, k} + ^{(i, l)}\!\!\epsilon_{i, j, k} = v_{\{i, j, k, l\}},
\end{equation}
 and we have a contradiction.  Therefore, we have $\epsilon_{i, j, k} = 0$ and so $\delta_{i, j, k}' = \delta_{i, j, k}$ for all $i, j, k$, implying the desired equality $\widetilde{N}' = \widetilde{N}$.  Part (a) is proved.

Now the fact that $\widetilde{N}$ is normal in $\varpi_{8 \to 2}^{-1}(\mathfrak{S}_{2g + 1})$ implies that the conjugation action induces an action of $\varpi_{8 \to 2}^{-1}(\mathfrak{S}_{2g + 1})$ on the quotient $\widetilde{M}$.  Since the commutator subgroup of $\varpi_8(\Gamma(2))$ is contained in $\widetilde{N}$, conjugation by any element of $\varpi_8(\Gamma(2))$ fixes each element of $\varpi_8(\Gamma(2))$ modulo $\widetilde{N}$; it follows that the induced action of $\varpi_{8 \to 2}^{-1}(\mathfrak{S}_{2g + 1})$ on the quotient $\widetilde{M}$ factors through $\varpi_{8 \to 2}^{-1}(\mathfrak{S}_{2g + 1}) \twoheadrightarrow \varpi_{8 \to 2}^{-1}(\mathfrak{S}_{2g + 1})/\varpi_8(\Gamma(2)) = \mathfrak{S}_{2g + 1}$.  The formula for the action given in the statement of (b) follows from the observations that for any $u \in \Sp(V)$ with $\sigma := \varpi_2(u) \in \mathfrak{S}_{2g + 1}$ and for each $a_{i, j} \in V$ as given by Hypothesis \ref{hyp lifts}(b), we have $u t_{a_{i, j}} u^{-1} = t_{u(a_{i, j})}$; that the image modulo $2$ of $u(a_{i, j})$ is $v_{\{\sigma(i), \sigma(j)\}} \in V / 2V$; and that (as is evident from Proposition \ref{prop N}(e)) we have $\tilde{\pi} \circ \varpi_8(t_{b}) = \overline{\lfloor \sigma(i), \sigma(j) \rfloor}$ for any $b \in V$ whose image modulo $2$ is $v_{\{\sigma(i), \sigma(j)\}} \in V / 2V$.

\end{proof}

\begin{cor} \label{cor contains tilde N}

Suppose that $H \subset \varpi_{8 \to 2}^{-1}(\mathfrak{S}_d)$ is a subgroup satisfying $\varpi_{8 \to 2}(H) = \mathfrak{S}_d$.  We have the (strict) containment $\widetilde{N} \subset H$.  In fact, we have that $H \not\supset \varpi_8(\Gamma(4))$ implies $H \cap \varpi_8(\Gamma(2)) = \widetilde{N}$; in this case, we have $\#H = (2g + 1)! \cdot 2^{4g^2 - 2g}$.

\end{cor}

\begin{proof}

Proposition \ref{prop N}(e) implies that the commutator of any element of $\varpi_8(\Gamma(2))$ with any element of $H \cap \varpi_8(\Gamma(2))$ lies in $N^{(4)}$, which by Lemma \ref{lemma contains N mod 8} is contained in $H \cap \varpi_8(\Gamma(2))$, so we have that $H \cap \varpi_8(\Gamma(2))$ is a normal subgroup of $\varpi_8(\Gamma(2))$.  Since $H \cap \varpi_8(\Gamma(2))$ is clearly normal in $H$ and since $\varpi_{8 \to 2}^{-1}(\mathfrak{S}_{2g + 1})$ is generated by its subgroups $H$ and $\varpi_8(\Gamma(2))$, we have $H \cap \varpi_8(\Gamma(2)) \lhd \varpi_{8 \to 2}^{-1}(\mathfrak{S}_{2g + 1})$.

The second statement of the corollary clearly implies the first, so we now assume that $H$ does not contain $\varpi_8(\Gamma(4))$ and show that $H \cap \varpi_8(\Gamma(2)) = \widetilde{N}$.  Indeed, thanks to Lemma \ref{lemma sufficient}, our assumption implies that $\varpi_{8 \to 4}(H)$ does not contain $\varpi_4(\Gamma(2))$.  Then Lemmas \ref{lemma contains N} and \ref{lemma contains N mod 8} respectively show that $\varpi_{8 \to 4}(H) \cap \varpi_4(\Gamma(2)) = N^{(2)}$ and $H \cap \varpi_8(\Gamma(4)) = N^{(4)}$.  Thus, the uniqueness statement in Proposition \ref{prop tilde N}(a) implies that the subgroup $H \cap \varpi_8(\Gamma(2)) \lhd \varpi_{8 \to 2}^{-1}(\mathfrak{S}_{2g + 1})$ must coincide with $\widetilde{N}$.

The claimed cardinality of $H$ comes from the computation $\#H = \#\varpi_{8 \to 2}(H) \#(H \cap \varpi_8(\Gamma(2)) = 
\#S_{2g + 1} \#\widetilde{N} = (2g + 1)! \cdot \#N^{(2)} \#N^{(4)} = (2g + 1)! \cdot 2^{4g^2 - 2g}$.

\end{proof}

\begin{cor} \label{cor contains tilde N2}

Suppose that $H \subset \varpi_{8 \to 2}^{-1}(\mathfrak{S}_d)$ is a subgroup satisfying $\varpi_{8 \to 2}(H) = \mathfrak{S}_d$ and $H \cap \varpi_8(\Gamma(2)) = \widetilde{N}$.  For each $\sigma \in \mathfrak{S}_d$, let $u_{\sigma} \in H$ be an element such that $\varpi_{8 \to 2}(u_{\sigma}) = \sigma$.  Then the set of $u_{\sigma}$'s generates $H$.

\end{cor}

\begin{proof}

This is exactly the same as the proof of Corollary \ref{cor contains N} (except that we use Corollary \ref{cor contains tilde N}).

\end{proof}

\subsection{Quasi-cocycles at level $8$} \label{sec4.2}

We now set out to classify all possibilities for the image of $G \subset \Sp(V)$ modulo $8$ assuming that $G$ does not contain $\Gamma(4) \lhd \Sp(V)$.  In this subsection and the next, we will present a series of definitions, lemmas, propositions which are analogous to the ones in \S\ref{sec3.2} and \S\ref{sec3.4} with analogous proofs; the arguments will therefore be presented briefly and with references to the corresponding arguments for the results in those earlier subsections.  Proposition \ref{prop mod 16} below does not rely on any of the results presented in this subsection and the next, and the reader interested in our argument for why $G$ must contain $\Gamma(8)$ may therefore skip the rest of \S\ref{sec3}.

We retain our choices of vectors $a_{i, j} \in V$ and lifts $\tilde{\sigma} \in \Sp(V)$ which were fixed in \S\ref{sec3.2}.  For each $\textbf{c} = (c_2, ... , c_{2g + 1}) \in \ff_2^{2g}$, we recall the construction given by Theorem \ref{thm mod 4}(b) which yields a quasi-cocycle $\varphi_{\mathbf{c}} : \mathfrak{S}_{2g + 1} \to M^{(2)}$ and a corresponding subgroup $H_{\mathbf{c}} \subsetneq \varpi_{4 \to 2}^{-1}(\mathfrak{S}_{2g + 1})$.  For any $\textbf{c} \in \ff_2^{2g}$ and any $i, j \in \{1, ... , 2g + 1\}$, we have seen in Remark \ref{rmk quasi-cocycle} that the compliment of the subset $I \subset \{1, ... , 2g + 1\}$ such that $\varpi_{\mathbf{c}}((i, j)) = v_I$ consists of exactly one natural number which is either $i$ or $j$; we denote this number by $m_{\mathbf{c}, i, j} \in \{i, j\}$.  We fix, once and for all, a set of lifts $y_{\mathbf{c}, \sigma} \in \Sp(V)$ for all $\sigma \in \mathfrak{S}_{2g + 1}$ which satisfies the following hypothesis.

\begin{hyp} \label{hyp lifts mod 8}

For each $\mathbf{c} \in \ff_2$, we assume the following.

a) We have $y_{\mathbf{c}, (1)} = 1$ and $\pi \circ \varpi_4(y_{\mathbf{c}, \sigma} \tilde{\sigma}^{-1}) = \varphi_{\mathbf{c}}(\sigma) \in M^{(2)}$ for $\sigma \in \mathfrak{S}_d$.

b) For each transposition $(i, j) \in \mathfrak{S}_{2g + 1}$, we have 
$$y_{\mathbf{c}, (1, j)} = \widetilde{(m_{\mathbf{c}, i, j}, 1)}^2 ... \ \widetilde{(m_{\mathbf{c}, i, j}, m_{\mathbf{c}, i, j} - 1)}^2 \ \widetilde{(m_{\mathbf{c}, i, j}, m_{\mathbf{c}, i, j} + 1)}^2 ... \ \widetilde{(m_{\mathbf{c}, i, j}, 2g + 1)}^2 \ \widetilde{(i, j)}.$$

\end{hyp}

It is clear from the construction given by Theorem \ref{thm mod 4}(b) for $\varphi_{\mathbf{c}}$ and our definition for the $m_{\mathbf{c}, i, j}$'s that part (b) of the above hypothesis is compatible with part (a) and that therefore a set of such lifts does exist (again, lifts of nontrivial non-transpositions may be chosen arbitrarily).  In order to obtain an analog of Lemma \ref{lemma lift identities}, we first need another lemma.

\begin{lemma} \label{lemma a + c}

Let $a, b, c \in V$ be vectors whose images modulo $2$ form a linearly independent subset of $V / 2V$ and which satisfy $\langle a, b \rangle = \langle b, c \rangle = -1$ and $\langle a, c \rangle = 0$.  We have the identity 
\begin{equation} \label{eq a + c}
t_a^2 t_{t_b(a)}^2 t_{t_c(t_b(a))}^2 t_b^2 t_{t_c(b)}^2 t_c^2 = t_a^2 t_{a - b}^2 t_{a - b + c}^2 t_b^2 t_{b - c}^2 t_c^2 = t_{a + c}^2.
\end{equation}

\end{lemma}

\begin{proof}

This can be checked through a straightforward but tedious computation which we outline as follows. We first note that the submodule $W \subset V$ consisting of vectors $v$ satisfying $\langle v, a \rangle = \langle v, c \rangle = 0$ is of rank $2g - 2$ as it is the orthogonal compliment of the subspace generated by $a$ and $c$ (which by linear independence modulo $2$ and Nakayama's Lemma has rank $2$).  Then it is clear that there is a vector $d \in V$ satisfying $\langle d, a \rangle = 1$ and $\langle d, c \rangle = 0$ and that $V$ is generated over its submodule $W$ by the elements $b, d$.  Thus it suffices to verify that the operator on the left-hand side of (\ref{eq a + c}) acts as the identity on $W$, fixes $b$, and sends $d$ to $d + a + c$; it is straightforward to check this by appropriately applying the formulas defining the transvections appearing on the left-hand side to an arbitrary $w \in W$ and the vectors $b$ and $d$.

\end{proof}

\begin{rmk} \label{rmk topological}

The above lemma has a topological interpretation which the author used to first arrive at the formula in (\ref{eq a + c}).  The rough idea is to imagine a compact genus-$g$ Riemann surface $\Sigma$ with a degree-$2$ covering map to the Riemann sphere $S$, ramified at the points $1, ... , 2g + 1, \infty \in S$, and to identify $V$ with $H_1(\Sigma, \zz) \otimes \zz_2$ in such a way that the vectors $a, b, c \in V$ correspond to classes of simple loops on $\Sigma$ whose images on $S$ wrap counterclockwise around the subsets $\{1, 2\}, \{2, 3\}, \{3, 4\}$ of the ramification points respectively.  There is an action of the planar braid group $B_{2g + 1}$ on $H_1(\Sigma, \zz)$ given by the \textit{reduced integral Burau representation} as described in \cite[\S2.1]{brendle2018level} (which is also induced by the action $\rho$ on the fundamental group of $S \smallsetminus \{1, ..., 2g + 1, \infty\}$ given in \cite[\S2.1]{hasson2020prime}).  This action is such that for any even-cardinality subinterval $\{2r + 1, ... , 2s\} \subset \{1, ... , 2g + 1\}$, the (pure) braid $\alpha_I \in B_{2g + 1}$ given by rotating the points in $I$ counterclockwise in a full circle acts on $H_1(\Sigma, \zz)$ as $t_{a_I}^2$ where $a_I \in H_1(\Sigma, \zz)$ is represented by a simple loop whose image on $S$ wraps counterclockwise around the subset $I$ of ramification points.  Meanwhile, the standard generators $\beta_i$ of $B_{2g + 1}$ for $1 \leq i \leq 2g$ (where $\beta_i$ rotates the points $i$ and $i + 1$ in a counterclockwise semicircular motion) each act as $t_{a_{\{i, i + 1\}}}$.  Then the desired identity follows from the visual verification that $\alpha_{\{1, 2, 3, 4\}}$ equals the composition $\alpha_{\{1, 2\}} (\beta_2\alpha_{\{1, 3\}}\beta_2^{-1}) (\beta_3\beta_2\alpha_{\{1, 4\}} \beta_2^{-1}\beta_3^{-1}) \alpha_{\{2, 3\}} (\beta_3\alpha_{\{2, 4\}}\beta_3^{-1}) \alpha_{\{3, 4\}}$.  There are also topological interpretations of a similar flavor for Lemmas \ref{lemma compliment} and \ref{lemma -1}.

\end{rmk}

\begin{lemma} \label{lemma lift identities mod 8}

With respect to any $\mathbf{c} \in \ff_2^{2g}$, the lifts of transpositions fixed above satisfy the relations 
\begin{equation} \label{eq i j k mod 8}
\tilde{\pi} \circ \varpi_8(y_{\mathbf{c}, (i, j)}y_{\mathbf{c}, (i, k)}y_{\mathbf{c}, (i, j)}) = \tilde{\pi} \circ \varpi_8(y_{\mathbf{c}, (i, k)}y_{\mathbf{c}, (i, j)}y_{\mathbf{c}, (i, k)})
\end{equation}
 and 
\begin{equation} \label{eq i j k l mod 8}
\tilde{\pi} \circ \varpi_8(y_{\mathbf{c}, (i, j)}y_{\mathbf{c}, (k, l)}) = \tilde{\pi} \circ \varpi_8(y_{\mathbf{c}, (k, l)}y_{\mathbf{c}, (i, j)})
\end{equation}
 for distinct $i, j, k, l$.

\end{lemma}

\begin{proof}

We start by noting that the first identity in Lemma \ref{lemma lift identities} implies that $\tilde{\pi} \circ \varpi_8(\widetilde{(i, j)}\widetilde{(i, k)}\widetilde{(i, j)}) = \tilde{\pi} \circ \varpi_8(\widetilde{(i, k)}\widetilde{(i, j)}\widetilde{(i, k)})$.  Meanwhile, let us first assume that $m_{\mathbf{c}, i, j} = m_{\mathbf{c}, i, k} = i$.  Then we have 
$$\tilde{\pi} \circ \varpi_8(y_{\mathbf{c}, (i, j)}y_{\mathbf{c}, (i, k)}y_{\mathbf{c}, (i, j)}\widetilde{(i, j)}^{-1}\widetilde{(i, k)}^{-1}\widetilde{(i, j)}^{-1}) = \sum_{m \neq i} \lfloor i, m \rfloor + ^{(i, j)}\!\!\sum_{m \neq i} \lfloor i, m \rfloor + ^{(i, j)(i, k)}\!\!\sum_{m \neq i} \lfloor i, m \rfloor$$
\begin{equation}
 = \sum_{l = i, j, k; \ m \neq i, j, k} \lfloor l, m \rfloor.
\end{equation}
  Since this expression is invariant under transposition of $j$ and $k$, we get the desired equality 

\noindent $\tilde{\pi} \circ \varpi_8(y_{\mathbf{c}, (i, j)}y_{\mathbf{c}, (i, k)}y_{\mathbf{c}, (i, j)}) = \tilde{\pi} \circ \varpi_8(y_{\mathbf{c}, (i, k)}y_{\mathbf{c}, (i, j)}y_{\mathbf{c}, (i, k)})$.  In the other cases (where $(m_{\mathbf{c}, i, j}, m_{\mathbf{c}, i, k})$ is $(i, k)$, $(j, i)$, or $(j, k)$), the first identity in the statement results from a similarly straightforward calculation.

To prove the second identity, we first note that the element $\tilde{\pi} \circ \varpi_8(y_{\mathbf{c}, (i, j)}\widetilde{(i, j)}^{-1})$ is fixed by the element $(k, l) \in \mathfrak{S}_{2g + 1}$ under the $\mathfrak{S}_{2g + 1}$-action on $\widetilde{M}$, and therefore conjugating $\tilde{\pi} \circ \varpi_8(y_{\mathbf{c}, (i, j)})$ by $\tilde{\pi} \circ \varpi_8(y_{\mathbf{c}, (k, l)})$ yields $\tilde{\pi} \circ \varpi_8(y_{\mathbf{c}, (i, j)} \widetilde{(i, j)}^{-1}\widetilde{(k, l)}\widetilde{(i, j)}\widetilde{(k, l)}^{-1})$.  It therefore suffices to show that $\tilde{\pi} \circ \varpi_8(\widetilde{(i, j)})$ and $\tilde{\pi} \circ \varpi_8(\widetilde{(k, l)})$ commute.  Clearly we have $\langle a_{i, j}, a_{k, l} \rangle \in 2\zz_2$.  One computes directly that $t_{a_{i, j}}t_{a_{k, l}}t_{a_{i, j}}^{-1}t_{a_{k, l}}^{-1}$ is equivalent modulo $8$ to $t_{a_{i, j} + a_{k, l}}^4 t_{a_{i, j}}^{4} t_{a_{k, l}}^{4}$ (resp. $t_{a_{i, j} + a_{k, l}}^{\pm 2} t_{a_{i, j}}^{\pm 2} t_{a_{k, l}}^{\pm 2}$) if $\langle a_{i, j}, a_{k, l} \rangle \equiv 0$ (mod $4$) (resp. if $\langle a_{i, j}, a_{k, l} \rangle \equiv \pm 2$ (mod $8$)), by appropriately composing the formulas defining the transvections involved and comparing the two sides.  In the first case, since $\varpi_8(t_{a_{i, j} + a_{k, l}}^4 t_{a_{i, j}}^{4} t_{a_{k, l}}^{4}) = 2(\lfloor i, k \rfloor + \lfloor i, l \rfloor + \lfloor j, k \rfloor + \lfloor j, l \rfloor) = \Delta_{i, j, k} + \Delta_{j, k, l} \in N^{(4)} \subset \widetilde{N}$, we get that $\tilde{\pi} \circ \varpi_8(t_{a_{i, j}}t_{a_{k, l}}t_{a_{i, j}}^{-1}t_{a_{k, l}}^{-1}) = 0$ and we are done.  We therefore assume that $\langle a_{i, j}, a_{k, l} \rangle \equiv 2$ (mod $4$).  We will use Lemma \ref{lemma a + c} to expand the image of the term $t_{a_{i, j} + a_{k, l}}^4$.  Set $a = a_{i, j}$ and set $c$ to be $a_{k, l}$ minus $\langle a_{i, j}, a_{k, l} \rangle$ times some vector $d \in V$ with $\langle a, d \rangle = 1$, so that $\langle a, c \rangle = 0$.  Since we have $\langle a, a_{j, k} \rangle, \langle a_{j, k}, c \rangle \in \zz_2^{\times}$, after multiplying $a_{j, k}$ by a suitable scalar and then adding a suitable multiple of $a_{k, l}$ to it, we get a vector $b \in V$ with $\langle a, b \rangle = \langle b, c \rangle = -1$.  Using Proposition \ref{prop N}(f) and the fact that $a \equiv a_{i, j}$, $b \equiv a_{j, k}$, and $c \equiv a_{k, l}$ (mod $2$), we see that $\tilde{\pi} \circ \varpi_8(t_{a_{i, j} + a_{k, l}}^{\pm 2} t_{a_{i, j}}^{\pm 2} t_{a_{k, l}}^{\pm 2}) = \tilde{\pi} \circ \varpi_8(t_{a + c}^{\pm 2} t_a^{\pm 2} t_c^{\pm 2})$.  Now we apply Lemma \ref{lemma a + c} to get 
$$\tilde{\pi} \circ \varpi_8(t_{a + c}^{\pm 2} t_a^{\pm 2} t_c^{\pm 2}) = \pm(\overline{\lfloor i, j \rfloor} + \overline{\lfloor i, k \rfloor} + \overline{\lfloor i, l \rfloor} + \overline{\lfloor j, k \rfloor} + \overline{\lfloor j, l \rfloor} + \overline{\lfloor k, l \rfloor} + \overline{\lfloor i, j \rfloor} + \overline{\lfloor k, l \rfloor})$$
\begin{equation}
= \pm \tilde{\pi}(\delta_{i, j, k} \delta_{i, j, l}) \pm \tilde{\pi}(\sum_{m \neq i, j, k, l} \Delta_{k, l, m}) = 0.
\end{equation}
  We have thus shown that the commutator of $\tilde{\pi} \circ \varpi_8(\widetilde{(i, j)})$ and $\tilde{\pi} \circ \varpi_8(\widetilde{(k, l)})$ is trivial, and the second identity follows.

\end{proof}

We now introduce a modulo-$8$ analog of a quasi-cocycle.

\begin{dfn} \label{dfn quasi-cocycle mod 8}

A \textit{level-$8$ quasi-cocycle} (of type $\mathbf{c}$) is a map $\varphi : \mathfrak{S}_{2g + 1} \to M^{(4)}$ satisfying the condition that 
$$\varphi(\sigma \tau) = \varphi(\sigma) + ^{\sigma}\!\!\varphi(\tau) + \tilde{\pi} \circ \varpi_8(y_{\mathbf{c}, \sigma}y_{\mathbf{c}, \tau} y_{\mathbf{c}, \sigma\tau}^{-1})$$
 for some $\mathbf{c} \in \ff_2^{2g}$.

\end{dfn}

We see that level-$8$ quasi-cocycles correspond to the subgroups $H \subset \varpi_{8 \to 2}^{-1}(\mathfrak{S}_{2g + 1})$ we are looking for, analogously to the situation modulo $4$.

\begin{lemma} \label{lemma quasi-cocycle mod 8}

Suppose that $H \subset \varpi_{8 \to 2}^{-1}(\mathfrak{S}_{2g + 1})$ is a subgroup generated by a set $\{u_{\sigma}\}_{\sigma \in \mathfrak{S}_{2g + 1}}$ with $\varpi_{8 \to 2}(u_{\sigma}) = \sigma$.  Assume that we have $\varpi_{8 \to 4}(H) \cap \varpi_4(\Gamma(2)) = N^{(2)}$, so that the subgroup $\varpi_{8 \to 4}(H) \subset \varpi_{4 \to 2}^{-1}(\mathfrak{S}_{2g + 1})$ corresponds via Theorem \ref{thm mod 4}(b) to the quasi-cocycle $\varphi_{\mathbf{c}} : \mathfrak{S}_{2g + 1} \to M^{(2)}$ for some $\mathbf{c} \in \ff_2^{2g}$.  Let $\varphi : \mathfrak{S}_{2g + 1} \to M^{(4)}$ be the map given by $\sigma \mapsto \tilde{\pi}(u_{\sigma} \varpi_8(y_{\mathbf{c}, \sigma})^{-1})$.  Then we have $H \cap \varpi_8(\Gamma(2)) = \widetilde{N}$ if and only if $\varphi$ is a level-$8$ quasi-cocycle of type $\mathbf{c}$.  In fact, there is a one-to-one correspondence between level-$8$ quasi-cocycles of type $\mathbf{c}$ and such subgroups $H \subset \varpi_{4 \to 2}^{-1}(\mathfrak{S}_d)$: this correspondence is given by sending a level-$8$ quasi-cocycle $\varphi : \mathfrak{S}_{2g + 1} \to M^{(4)}$ to the subgroup generated by the elements $\phi_{\sigma} \varpi_8(y_{\mathbf{c}, \sigma})$, where $\phi_{\sigma} \in \varpi_8(\Gamma(2))$ is any element such that $\tilde{\pi}(\phi_{\sigma}) = \varphi(\sigma)$.

\end{lemma}

\begin{proof}

This argument is precisely analogous to the one used to prove Lemma \ref{lemma quasi-cocycle} (here we use Corollaries \ref{cor contains tilde N} and \ref{cor contains tilde N2}).

\end{proof}

We now present an analog of Lemma \ref{lemma quasi-cocycle2} (note the difference in property (i) of each statement).

\begin{lemma} \label{lemma quasi-cocycle2 mod 8}

Let $\varphi : \mathfrak{S}_{2g + 1} \to M^{(4)}$ be a level-$8$ quasi-cocycle.  Then the map $\varphi$ satisfies the following conditions.

(i) We have $\langle \varphi((i, j)), v_{\{i, j\}} \rangle = g$ for $1 \leq i < j \leq 2g + 1$.

(ii) We have $\langle \varphi((i, j)), v_{\{k, l\}} \rangle = 0$ for distinct $i, j, k, l$.

Conversely, suppose that we define a map $\varphi : \mathfrak{S}_{2g + 1} \to M^{(4)}$ as follows.  We first assign values of $\varphi((1, j))$ for $2 \leq j \leq d$ which satisfy conditions (i) and (ii) and then, after fixing a vector $\mathbf{c} \in \ff_2^{2g}$, for each $\sigma \in \mathfrak{S}_{2g + 1} \smallsetminus \{(1, j)\}_{2 \leq j \leq d}$ we write $\sigma$ as a product of the generators $(1, j)$ and apply the condition given in Definition \ref{dfn quasi-cocycle} to determine $\varphi(\sigma)$.  Then the map $\varphi$ is well defined (i.e. it does not depend on the choice of presentation of each $\sigma$ as a product of generators), and $\varphi$ is a quasi-cocycle of type $\mathbf{c}$.

\end{lemma}

\begin{proof}

As with level-$4$ quasi-cocycles, we know that any level-$8$ quasi-cocycle takes the trivial permutation to $0 \in M^{(4)}$, and by a similar argument as in the beginning of the proof of Lemma \ref{lemma quasi-cocycle2}, for any transposition $(i, j) \in \mathfrak{S}_{2g + 1}$ we get 
\begin{equation}
0 = \ ^{\sigma}\!\varphi(\sigma) + \varphi(\sigma) + \tilde{\pi} \circ \varpi_8(y_{\mathbf{c}, (i, j)}^2) = \langle \varphi(\sigma), v_{\{i, j\}} \rangle v_{\{i, j\}} + \tilde{\pi} \circ \varpi_8(y_{\mathbf{c}, (i, j)}^2).
\end{equation}
  It follows from Proposition \ref{prop tilde N}(a) that we may identify $M^{(4)}$ with a subgroup of $\widetilde{M}$.  We proceed to show that $\tilde{\pi} \circ \varpi_8(y_{\mathbf{c}, (i, j)}^2) = gv_{\{i, j\}} \in M^{(4)}$, from which it follows that property (i) holds.  Now an easy calculation shows that 
$$\tilde{\pi} \circ \varpi_8(y_{\mathbf{c}, (i, j)}^2) = \sum_{k \neq i} \overline{\lfloor i, k \rfloor} + \sum_{k \neq j} \overline{\lfloor j, k \rfloor} + \overline{\lfloor i, j \rfloor} = -\overline{\lfloor i, j \rfloor} + \sum_{k \neq i, j} (\overline{\lfloor i, k \rfloor} + \overline{\lfloor j, k \rfloor})$$
\begin{equation}
 = 2g \overline{\lfloor i, j \rfloor} + \sum_{k \neq i, j} \tilde{\pi}_8(\delta_{i, j, k}) + \sum_{\{l, m\} \cap \{i, j\} = \varnothing} 2\overline{\lfloor l, m \rfloor} = 2g \overline{\lfloor i, j \rfloor} = gv_{\{i, j\}} \in M^{(4)}.
\end{equation}
  (For the second-to-last equality we have used the fact that $\sum_{\{l, m\} \cap \{i, j\} = \varnothing} [l, m] \in N^{(4)}$ by definition and therefore the element $\prod_{\{l, m\} \cap \{i, j\} = \varnothing} \lfloor l, m \rfloor^2 \in \varpi_8(\Gamma(4))$ lies in the kernel of $\tilde{\pi}$.)

The rest of the proof of this lemma is precisely analogous to the proof of Lemma \ref{lemma quasi-cocycle2} (here we use Lemma \ref{lemma lift identities mod 8}).

\end{proof}

\subsection{The possible images of $G$ modulo $8$} \label{sec4.3}

The following theorem provides us with the full story modulo $8$; along with the cardinality result in Corollary \ref{cor contains tilde N}, it implies all of Theorem \ref{thm main}(b) except for the assertion that $G$ contains $\Gamma(8) \lhd \Sp(V)$ (which we show in the next section).

\begin{thm} \label{thm mod 8}

Fix any vector $\mathbf{c} \in \ff_2^{2g}$.  For each $\mathbf{d} = (d_2, ... , d_{2g + 1}) \in \ff_2^{2g}$, let $\varphi_{\mathbf{c}, \mathbf{d}} : \mathfrak{S}_{2g + 1} \to M^{(4)}$ be the map defined by assigning $\varphi_{\mathbf{c}, \mathbf{d}}((1, j)) = d_jv_{\{1, j\}} + g\sum_{2 \leq k \leq 2g + 1} v_{1, k}$ and determining $\varphi_{\mathbf{c}, \mathbf{d}}(\sigma)$ for $\sigma \in \mathfrak{S}_{2g + 1} \smallsetminus \{(1, j)\}_{2 \leq j \leq 2g + 1}$ as in the second statement of Lemma \ref{lemma quasi-cocycle2 mod 8} with respect to $\mathbf{c}$ (by that lemma, this map is a well-defined level-$8$ quasi-cocycle of type $\mathbf{c}$).  Then every level-$8$ quasi-cocycle $\varphi : \mathfrak{S}_{2g + 1} \to M^{(4)}$ of type $\mathbf{c}$ is equal to $\varphi_{\mathbf{c}, \mathbf{d}}$ for some $\mathbf{d} \in \ff_2^{2g}$.

Therefore, there are precisely $2^{4g}$ subgroups $H \subset \varpi_{8 \to 2}^{-1}(\mathfrak{S}_{2g + 1})$ satisfying $\varpi_{8 \to 2}(H) = \mathfrak{S}_{2g + 1}$ and $H \cap \varpi_8(\Gamma(2)) = \widetilde{N}$.  Moreover, the set of these subgroups forms a full conjugacy class of subgroups of $\varpi_{8 \to 2}^{-1}(\mathfrak{S}_{2g + 1})$.

\end{thm}

\begin{proof}

One proves that each level-$8$ quasi-cocycle of type $\mathbf{c}$ is of the form $\varphi_{\mathbf{c}, \mathbf{d}}$ in a precisely analogous way to the way we proved the corresponding statement in Theorem \ref{thm mod 4}(b), this time using Lemmas \ref{lemma quasi-cocycle mod 8} and \ref{lemma quasi-cocycle2 mod 8}.

We now prove the statement about conjugacy.  As in the proof of the analogous statement in Theorem \ref{thm mod 4}(b), it is clear that conjugating any subgroup $H \subset \varpi_{8 \to 2}^{-1}(\mathfrak{S}_{2g + 1})$ satisfying $\varpi_{8 \to 2}(H) = \mathfrak{S}_{2g + 1}$ and $H \cap \varpi_8(\Gamma(2)) = \widetilde{N}$ by any element in $\varpi_{8 \to 2}^{-1}(\mathfrak{S}_{2g + 1})$ yields another subgroup with those same properties.  It remains to show that all such subgroups are conjugate.  For each $\mathbf{c}, \mathbf{d} \in \ff_2^{2g}$, we denote the subgroup of $\varpi_2^{-1}(\mathfrak{S}_d)$ corresponding to $\varphi_{\mathbf{c}, \mathbf{d}}$ via Lemma \ref{lemma quasi-cocycle} by $H_{\mathbf{c}, \mathbf{d}}$.  Choose any $\mathbf{c}, \mathbf{d}, \mathbf{c}', \mathbf{d}' \in \ff_2^{2g}$.  It follows from what we have shown in (the proof of) Theorem \ref{thm mod 4}(b) that there is some element $u \in \varpi_4(\Gamma(2))$ such that $u H_{\mathbf{c}'} u^{-1} = H_{\mathbf{c}}$ (where $H_{\mathbf{c}}, H_{\mathbf{c}'} \subset \varpi_{4 \to 2}^{-1}(\mathfrak{S}_{2g + 1})$ are the subgruops denoted as such in that proof).  It follows that for any lift $\tilde{u} \in \varpi_8(\Gamma(2))$ with $\varpi_{8 \to 4}(\tilde{u}) = u$, we have $\tilde{u} H_{\mathbf{c}', \mathbf{d}'} \tilde{u}^{-1} = H_{\mathbf{c}, \mathbf{d}''}$ for some $\mathbf{d}'' \in \ff_2^{2g}$.  It therefore suffices to prove conjugacy under the assumption that $\mathbf{c} = \mathbf{c}'$; in other words, we only need to show for a fixed $\mathbf{c} \in \ff_2^{2g}$ that all subgroups $H \subset \varpi_{8 \to 2}^{-1}(\mathfrak{S}_{2g + 1})$ satisfying $\varpi_{8 \to 4}(H) = H_{\mathbf{c}}$ and $H \cap \varpi_8(\Gamma(4)) = N^{(4)}$ are conjugate.  To do this, we fix $\mathbf{c}, \mathbf{d} \in \ff_2^{2g}$ and we claim that for any choice of distinct $i, j \in \{1, ..., 2g + 1\}$, if $\mathbf{d}' \in \ff_2^{2g}$ is the vector such that $H_{\mathbf{c}, \mathbf{d}'} = \varpi_8(\widetilde{(i, j)}^4) H_{\mathbf{c}, \mathbf{d}} \varpi_8(\widetilde{(i, j)}^4)^{-1}$, then we have $\mathbf{d}' - \mathbf{d} = (\langle c_{i, j}, c_{1, k} \rangle)_{2 \leq k \leq 2g + 1}$.  This claim is shown using the same argument as was used for the analogous claim in the proof of Theorem \ref{thm mod 4}(b) with the terms $\widetilde{(i, j)}^2$, $\widetilde{(1, k)}$, $\pi$, $M^{(2)}$, and $N^{(2)}$ replaced by $\widetilde{(i, j)}^4$, $y_{\mathbf{c}, (1, k)}$, $\tilde{\pi}$, $M^{(4)}$, and $N^{(4)}$ respectively.  Now the desired conjugacy statement follows by the same nondegeneracy argument as in the end of the proof of Theorem \ref{thm mod 4}(b).

\end{proof}

\section{Lifting to $\zz / 16\zz$} \label{sec5}

In this section we show that the only possible modulo-$16$ images of a subgroup $G \subseteq \Sp(V)$ with $\varpi_2(G) = \mathfrak{S}_{2g + 1}$ are the inverse images under $\varpi_{16 \to 8}$ of the possible modulo-$8$ images found in the previous section, i.e. that we must have $\varpi_{16}(G) \supset \varpi_{16}(\Gamma(8))$.  By Lemma \ref{lemma sufficient}, this is equivalent to saying that we always have $G \supset \Gamma(8) \lhd \Sp(V)$, so in showing this we will complete the proof of Theorem \ref{thm main}(b).  We first need the following (easier) variant of Lemma \ref{lemma a + c}.

\begin{lemma} \label{lemma -1}

Given any vectors $a, b \in V$ with $\langle a, b \rangle \in \{\pm 1\}$, the map $t_a^2 t_{t_b(a)}^2 t_{b}^2 : V \to V$ is the order-$2$ linear automorphism which acts on the subspace generated by $\{a, b\}$ as the scalar $-1$ and on the orthogonal compliment of that subspace as the identity.

\end{lemma}

\begin{proof}

This can be verified directly from the formulas defining the transvections $t_a$ and $t_b$ by checking that the map $t_a^2 t_{t_b(a)}^2 t_{b}^2$ sends $a$ to $-a$ and $b$ to $-b$ while fixing all $c \in V$ satisfying $\langle c, a \rangle = \langle c, b \rangle = 0$. \qed \phantom\qedhere

\end{proof}

\begin{prop} \label{prop mod 16}

Suppose that $H \subset \varpi_{16 \to 2}^{-1}(\mathfrak{S}_d)$ is a subgroup satisfying $\varpi_{16 \to 2}(H) = \mathfrak{S}_{2g + 1}$.  Then we have $\varpi_{16}(\Gamma(8)) \subset H$.

\end{prop}

\begin{proof}

We know from Corollary \ref{cor contains tilde N} that the image $\varpi_{16 \to 8}(H)$ contains $\widetilde{N}$.  In particular, if we choose lifts $\widetilde{\Delta}_{i, j, k} \in H$ satisfying $\varpi_{16 \to 8}(\widetilde{\Delta}_{i, j, k}) = \Delta_{i, j, k} \in N^{(4)} \subset \widetilde{N} \subset \varpi_{16 \to 8}(H)$, it is easy to see (using the fact that $\varpi_{16}(\Gamma(4))$ is abelian by Proposition \ref{prop N}(e)) that the squares $\widetilde{\Delta}_{i, j, k}^2$ of these lifts are the generators $\Delta_{i, j, k}$ of $N^{(8)}$ (and also that given any $\tilde{s} \in \varpi_{16 \to 8}^{-1}(N^{(4)})$, we have $\tilde{s}^2 \in N^{(8)}$).  This implies in particular that we have $N^{(8)} \subseteq H \cap \varpi_{16}(\Gamma(8))$, and so by Proposition \ref{prop N}(b) it suffices to show that this containment is strict.

Consider the element $\delta_{1, 2, 3} = \varpi_8(t_{a_{1, 2}}^2 t_{a_{1, 3}}^2 t_{a_{2, 3}}^2) \prod_{4 \leq l < m \leq 2g + 1} \varpi_8(t_{a_{l, m}}^4) \in \varpi_{8}(\Gamma(2))$ which lies in $\widetilde{N} \subset \varpi_{16 \to 8}(H)$ by definition.  We clearly have $t_{a_{1, 2}}(a_{1, 3}) \equiv a_{2, 3}$ (mod $2$).  Then Proposition \ref{prop N}(f) implies that we have $s := \varpi_8(t_{t_{a_{1, 2}}(a_{1, 3})}^{-2} t_{a_{2, 3}}^{2}) \in N^{(4)}$.  Now choose a lift $\widetilde{\delta}_{1, 2, 3} \in H$ with $\varpi_{16 \to 8}(\widetilde{\delta}_{1, 2, 3}) = \delta_{1, 2, 3}$, which we write as $\widetilde{\delta}_{1, 2, 3} = \varpi_{16}(t_{a_{1, 2}}^2 t_{a_{1, 3}}^2 t_{t_{a_{1, 2}}(a_{1, 3})}^2) \tilde{s} \prod_{4 \leq l < m \leq 2g + 1} \varpi_{16}(t_{a_{l, m}}^4)$ for some $\tilde{s} \in \varpi_{16}(\Gamma(4))$ with $\varpi_{16 \to 8}(\tilde{s}) = s$.  Since the commutator of $\varpi_{16}(t_{a_{1, 2}}^2 t_{a_{1, 3}}^2 t_{t_{a_{1, 2}}(a_{1, 3})}^2)$ with any element of $\varpi_{16}(\Gamma(4))$ lies in $N^{(8)}$ by Proposition \ref{prop N}(e), we see that 
\begin{equation} \label{eq mod 16}
\widetilde{\delta}_{1, 2, 3}^2 \equiv \varpi_{16}((t_{a_{1, 2}}^2 t_{a_{1, 3}}^2 t_{t_{a_{1, 2}}(a_{1, 3})}^2)^2) \tilde{s}^2 \prod_{4 \leq l < m \leq 2g + 1} \varpi_{16}(t_{a_{l, m}}^8) \ \ \ \ (\mathrm{mod} \ N^{(8)}).
\end{equation}
  As we have noted above, we also have $\tilde{s}^2 \in N^{(8)}$; moreover, Lemma \ref{lemma -1} implies that $(t_{a_{1, 2}}^2 t_{a_{1, 3}}^2 t_{t_{a_{1, 2}}(a_{1, 3})}^2)^2 = 1 \in \Sp(V)$.  Therefore, we have $\widetilde{\delta}_{1, 2, 3}^2 \in H \cap \varpi_{16}(\Gamma(8))$ and the right-hand side of the equivalence in (\ref{eq mod 16}) can be simplified to $\sum_{4 \leq l < m \leq 2g + 1} [l, m] \in \varpi_{16}(\Gamma(8))$.  Since $g \geq 2$, this element is not an empty sum, and so by definition it does not belong to the subgroup $N^{(8)} \subset \varpi_{16}(\Gamma(8))$ and we are done.

\end{proof}

\section{An application} \label{sec6}

We now present a corollary of Theorem \ref{thm main}(a), which is a purely elementary statement regarding roots of even-degree polynomials with full Galois group.

\begin{thm} \label{thm even-degree polynomial}

Let $K$ be a field of characteristic different from $2$; let $f(x) \in K[x]$ be a separable polynomial of even degree $d \geq 6$ whose Galois group is the full $S_d$; and let $L$ be the splitting field of $f$ over $K$.  Assume that the discriminant $\Delta$ of $f$ does not lie in $-K^2$.  Write $\alpha_1, ... , \alpha_d \in L$ for the roots of $f$.  Then the image modulo $(L^{\times})^2$ of the set of elements $\alpha_j - \alpha_i \in L^{\times}$ for $1 \leq i < j \leq d$ is independent in the multiplicative group $L^{\times} / (L^{\times})^2$.

\end{thm}

\begin{proof}

We first assume that $\sqrt{-1} \notin K$.  Then we must have $\sqrt{-1} \notin L$ as well, because otherwise the unique quadratic subextension $K(\sqrt{\Delta}) / K$ would coincide with $K(\sqrt{-1})$, which would violate the hypothesis that $\Delta \notin -K^2$.  It follows that $K(\sqrt{-1}) \cap L = K$, so $L(\sqrt{-1})$ is the splitting field of $f$ over $K(\sqrt{-1})$ and the Galois group of $f$ over $K(\sqrt{-1})$ is still $S_d$.  If the statement of the theorem is true when $K$ is replaced by $K(\sqrt{-1})$, then it clearly holds over $K$ as well (since independence modulo $(L(\sqrt{-1})^{\times})^2$ is stronger than independence modulo $(L^{\times})^2$).  In light of this, we will assume for the rest of the proof that $\sqrt{-1} \in K$.

The statement of the theorem amounts to saying that no product of elements in a nonempty subset of $\{\alpha_j - \alpha_i\}_{1 \leq i < j \leq d}$ is a square in $L$.  We first show that $\pm \prod_{1 \leq i < j \leq d} (\alpha_j - \alpha_i)$ is not a square in $L$.  We first observe that the square of this product is the discriminant $\Delta$ of $f$, so the claim amounts to saying that $L$ does not contain a 4th root of $\Delta$.  Suppose that $\sqrt[4]{\Delta} \in L$.  Since $\Gal(L / K)$ is not contained in $A_d$, we have that $\sqrt{\Delta} \notin K$.  It follows that the extension $K(\sqrt[4]{\Delta}) / K$ has Galois group isomorphic to $\zz / 4\zz$, which does not appear as a quotient of $\Gal(L / K) = S_d$, so we have a contradiction.

For $1 \leq i, j \leq d - 1$, write $\gamma_{i, j} = (\alpha_j - \alpha_i) \prod_{l \neq i, j} (\alpha_d - \alpha_l)$, and let $L' = L(\{\sqrt{\gamma_{i, j}}\}_{1 \leq i < j \leq d - 1})$.
  We claim that if the $\alpha_i$'s are chosen ``generically" (i.e. if there is a field $k \subset K$ such that the $\alpha_i$'s are transcendental and independent over $k$, with $L = k(\{\alpha_i\}_{1 \leq i < j \leq d})$ and $K \subset L$ being the subfield fixed under all permutations of the $\alpha_i$'s), then there is a free $\zz_2$-module $V$ of rank $(d/2 - 1)$ equipped with a nondegenerate alternating bilinear pairing such that $\Gal(L' / K) = H := \varpi_{4 \to 2}^{-1}(\mathfrak{S}_d) \subset \Sp(V / 4V)$, with $\Gal(L' / L) \subset \Gal(L' / K)$ corresponding to $\varpi_4(\Gamma(2)) \subset H$.  In fact, it is possible to verify this claim directly, but in any case it follows from the description of the $4$-division field of the Jacobian of the hyperelliptic curve defined by $y^2 = f(x)$ in \cite[Theorem 2.4]{yelton2019abelian}, in which case $V$ is the $2$-adic Tate module $T_2$.  This implies that without the ``genericness" assumption, we have an inclusion $\Gal(L' / K) \hookrightarrow H$ such that the subgroup corresponding to $\Gal(L' / L)$ coincides with the intersection of $\Gal(L' / K) \subseteq H$ with $\varpi_4(\Gamma(2))$.  Since the image of $\Gal(L' / K) \subseteq H$ modulo $\varpi_4(\Gamma(2))$ is $\Gal(L / K) \subseteq \varpi_{4 \to 2}(H) = \mathfrak{S}_d$, which is isomorphic to $S_d$ by our assumption on the Galois group of $f$, we have $\varpi_2(\Gal(L' / K)) = \mathfrak{S}_d$.  Theorem \ref{thm main}(a) then implies that $\Gal(L' / K)$ contains $\varpi_4(\Gamma(2)) \subset H$ (or equivalently, that $\Gal(L' / K) = H$).  It follows that $\Gal(L' / L) = \varpi_4(\Gamma(2))$, which in turn is isomorphic to $(\zz / 2\zz)^{(d - 1)(d - 2) / 2}$ by \cite[Corollary 2.2]{sato2010abelianization} and its proof.  Thus, the extension $L'$ is generated over $L$ by the square roots of $(d - 1)(d - 2) / 2$ elements of $L^{\times}$ which are independent modulo $(L^{\times})^2$.  It then follows from the definition of $L'$ that the $(d - 1)(d - 2) / 2$-element set $\{\gamma_{i, j}\}_{1 \leq i < j \leq d - 1}$ is independent modulo $(L^{\times})^2$.

Now suppose that we are given a subset $\mathfrak{I'} \subseteq \mathfrak{I} := \{(i, j) \in \{1, ... , d\}^2 \ | \ i < j\}$ such that $\prod_{(i, j) \in \mathfrak{I}} (\alpha_j - \alpha_i) = a^2$ for some $a \in L$.  Note that for any permutation $\sigma \in S_d = \Gal(L / K)$, the element $\prod_{(i, j) \in \mathfrak{I}} (\alpha_{\sigma(j)} - \alpha_{\sigma(i)}) = (^{\sigma}\!a)^2$ lies in $L^2$.  We will now show that $\mathfrak{I}' = \varnothing$, which directly implies the statement of the theorem.  Suppose that $\mathfrak{I}' \neq \varnothing$.  We showed above that $\prod_{1 \leq i < j \leq d} (\alpha_j - \alpha_i) \notin L^2$, which implies that $\mathfrak{I}' \neq \mathfrak{I}$.  Therefore, there exist distinct natural numbers $q, r, s \in \{1, ... , d\}$ such that $(q, r) \in \mathfrak{I}'$ but $(q, s) \notin \mathfrak{I}'$ (here we are assuming without loss of generality that $q < r, s$).  Then, letting $\sigma$ be the transposition $(r, s)$, we have $\prod_{(i, j) \in \mathcal{I}'} (\alpha_j - \alpha_i)(\alpha_{\sigma(j)} - \alpha_{\sigma(i)}) \in L^2$.  It is straightforward to check that 
\begin{equation} \label{eq 1.1}
1 \equiv \prod_{(i, j) \in \mathfrak{I}'} (\alpha_j - \alpha_i)(\alpha_{\sigma(j)} - \alpha_{\sigma(i)}) \equiv \prod_{(i, j) \in I \times \{r, s\}} (\alpha_j - \alpha_i) \ \ \ \ (\mathrm{mod} \ (L^{\times})^2)
\end{equation}
 for some subset $I \subset \{1, ... , d\} \smallsetminus \{r, s\}$ with $q \in I$.

First suppose that $I$ has even cardinality.  Then it follows from a straightforward computation that the element on the right-hand side of the equivalence (\ref{eq 1.1}) is equivalent modulo $(L^{\times})^2$ to $\prod_{(i, j) \in I \times \{r, s\}} \gamma_{i, j}$.  But (\ref{eq 1.1}) says that this is equivalent to $1$, which contradicts the independence of the $\gamma_{i, j}$'s.  Now suppose that $I$ has odd cardinality.  Then $I \subsetneq \{1, ... , d\} \smallsetminus \{r, s\}$.  Choose $t \in \{1, ... , d\} \smallsetminus (\{r, s\} \cup I)$, and let $\tau \in S_d$ be the transposition $(q, t)$.  We then have 
\begin{equation} \label{eq 1.2}
1 \equiv \prod_{(i, j) \in I \times \{r, s\}} (\alpha_j - \alpha_i)(\alpha_{\tau(j)} - \alpha_{\tau(i)}) \equiv \prod_{(i, j) \in \{q, t\} \times \{r, s\}} (\alpha_j - \alpha_i) \ \ \ \ (\mathrm{mod} \ (L^{\times})^2).
\end{equation}
  Now similarly, the element on the right-hand side is equivalent modulo $(L^{\times})^2$ to $\prod_{(i, j) \in \{q, t\} \times \{r, s\}} \gamma_{i, j}$, and (\ref{eq 1.2}) contradicts the independence of the $\gamma_{i, j}$'s.

\end{proof}

\bibliographystyle{plain}
\bibliography{bibfile}

\end{document}